\documentclass{article}
\usepackage[top=2cm,bottom=2cm, left=2cm, right=3cm]{geometry} 
\usepackage{amsmath, amssymb, amsthm}
\usepackage[utf8]{inputenc}

\usepackage{parskip}
\usepackage{tikz}
\usepackage{tikz}
\usepackage{tikz-cd}
\usepackage{url}

\usepackage{hyperref}

\DeclareMathSymbol{\shortminus}{\mathbin}{AMSa}{"39}

\newcommand{\sq}{\mathrm{sq}}
\newcommand{\ex}{\mathrm{ex}}
\newcommand{\dR}{\mathrm{dR}}
\newcommand{\cl}{\mathrm{cl}}
\newcommand{\fX}{\mathfrak{X}}

\newcommand{\N}{\mathbb{N}}
\newcommand{\Z}{\mathbb{Z}}
\newcommand{\R}{\mathbb{R}}

\newcommand{\Lie}{\mathrm{Lie}}
\newcommand{\Div}{\mathrm{div}}
\newcommand{\Diff}{\mathrm{Diff}}
\newcommand{\supp}{\mathrm{supp}}
\newcommand{\oline}{\overline}
\newcommand{\cD}{\mathcal{D}}
\newcommand{\cG}{\mathcal{G}}
\newcommand{\dd}{{\tt d}}

\newcommand{\fg}{\mathfrak{g}}

\newcommand{\fL}{\mathfrak{L}}
\newcommand{\U}{\mathrm{U}}
\newcommand{\ceX}{\fX_{\mathrm{c}\shortminus\mathrm{ex}}}

\newtheorem{Theorem}{Theorem}[section]
\newtheorem{Lemma}[Theorem]{Lemma}
\newtheorem{Proposition}[Theorem]{Proposition}
\newtheorem{Corollary}[Theorem]{Corollary}

\theoremstyle{definition}
\newtheorem{Definition}[Theorem]{Definition}
\newtheorem{Remark}[Theorem]{Remark}

\title{Universal central extension of the Lie algebra of exact divergence-free vector fields}
\author{
Bas Janssens\thanks{Delft University of Technology, The Netherlands. \texttt{B.Janssens@tudelft.nl}
},
Leonid Ryvkin\thanks{
Université Claude Bernard Lyon 1, France and University of Göttingen, Germany. \texttt{ryvkin@math.univ-lyon1.fr}
},
Cornelia Vizman\thanks{
West University of Timi\c soara, Romania. \texttt{cornelia.vizman@e-uvt.ro}
}
}

\begin{document}

\maketitle

\begin{abstract}
We construct the universal central extension of the Lie algebra of exact divergence-free vector fields, proving a conjecture by Claude Roger from 1995. The proof relies on the analysis of a Leibniz algebra that underlies these vector fields. 
As an application, we construct the universal central extension of the (infinite-dimensional) Lie group of exact divergence-free diffeomorphisms of a compact 3-dimensional manifold.
\end{abstract}

\tableofcontents

\newpage

\section{Introduction}

The Lie algebra $\fX(M,\mu)$ of divergence-free vector fields on a smooth manifold $M$ with a volume form $\mu$ is one of the four 
classical infinite-dimensional Lie algebras studied by \'{E}.\ Cartan~\cite{Cartan1909, SingerSternberg1965}, the other three being the Lie algebras of symplectic vector fields, of contact vector fields, and of all vector fields.
The goal of this article is to classify the continuous central extensions of $\fX(M,\mu)$, or, equivalently, to determine the second continuous Lie algebra cohomology $H^2(\fX(M,\mu),\R)$.

In order to do so, we study the Lie algebra $\fX_{\ex}(M,\mu)$ of exact divergence-free vector fields. 
The study of this Lie algebra goes back at least to~\cite{Lichnerowicz1974}, where it is identified as a perfect commutator ideal of the Lie algebra 
$\mathfrak X(M,\mu)$ of all divergence-free vector fields. This is analogous to the Lie algebra $\fX_{\mathrm{Ham}}(M,\omega)$ of Hamiltonian vector fields for a symplectic form $\omega$, which is a perfect commutator ideal in the Lie algebra of all symplectic vector fields \cite{ALDM74}.

The following central extension of $\fX_{\ex}(M,\mu)$ appears in Ismagilov~\cite{Ismagilov1980}, and is usually attributed to Lichnerowicz. 
It was conjectured to be universal by Roger~\cite{Roger1995}, and our main result confirms this conjecture.

Recall that a vector field $X$ on $M$ is divergence-free if $L_{X}\mu = 0$ or, equivalently, if $\iota_{X}\mu$ is closed. It is exact divergence-free if $\iota_{X}\mu$ is exact, and $\alpha \in \Omega^{n-2}(M)$ is called a potential for $X$ if $\iota_{X}\mu = d\alpha$. We denote the unique vector field with potential $\alpha$ by 
$X_{\alpha}$.
On the space $\oline{\Omega}{}^{n-2}(M) := \Omega^{n-2}(M)/d\Omega^{n-3}(M)$ of potentials modulo exact potentials, we define a Lie bracket by 
$[\oline{\alpha}, \oline{\beta}] = \oline{i_{X_{\alpha}} i_{X_{\beta}}\mu}$.
This makes $\oline{\alpha} \mapsto X_{\alpha}$
into a surjective Lie algebra homomorphism 
$\oline{\Omega}{}^{n-2}(M) \rightarrow \fX_{\ex}(M)$, 
and its kernel $H^{n-2}_{\dR}(M)$ is precisely the centre of $\oline{\Omega}{}^{n-2}(M)$.
Our main result is that for $\mathrm{dim}(M) \geq 3$,
 the central extension
\begin{equation}\label{hof1}
H^{n-2}_{\dR}(M) \rightarrow \oline{\Omega}{}^{n-2}(M) 
\rightarrow \fX_{\ex}(M,\mu)
\end{equation}
is universal 
in the category of locally convex Lie algebras. 
In fact we will prove the slightly stronger version for compactly supported potentials, from which the above result easily follows.

Proving that \eqref{hof1} is universal is of course equivalent to showing that $H^2(\fX_{\ex}(M,\mu), \R) = H_{n-2}(M,\R)$.
For the Lie algebra $\fX(M,\mu)$ of all divergence-free vector fields (which does not, in general, admit a universal extension),
one then obtains 
$H^2(\fX(M,\mu), \R) = H_{n-2}(M,\R) \oplus \wedge^2 H_{n-1}(M,\R)$.

Compared with the universal central extension of $\mathfrak X_{\mathrm{Ham}}(M,\omega)$ in our previous work \cite{JV16}, 
one of the main difficulties of working with $\mathfrak X_{\ex}(M,\mu)$ is that neither $\mathfrak X_{\ex}(M,\mu)$ nor 
$\oline{\Omega}{}^{n-2}(M)$ are projective modules over the ring $C^{\infty}(M)$ of smooth functions. This makes it difficult to work with differential operators. 
We resolve this problem by shifting focus to the projective $C^{\infty}(M)$-module $\Omega^{n-2}(M)$.
Although this is no longer a Lie algebra, it is still a (left) Leibniz algebra with the bracket $[\alpha,\beta] = L_{X_{\alpha}}\beta$, and as a Leibniz algebra, it has natural projections to $\oline{\Omega}{}^{n-2}(M)$ and $\mathfrak X_{\ex}(M,\mu)$.

If $M$ is compact, then the group $\Diff_{\ex}(M,\mu)$ of exact volume preserving diffeomorphisms is a Fr\'echet--Lie group with Lie algebra $\fX_{\ex}(M,\mu)$. 
Although we will not touch on this topic in the present paper,
the classification of continuous central extensions of $\fX_{\ex}(M,\mu)$ is intricately linked to projective unitary representation theory of $\Diff_{\ex}(M,\mu)$. Indeed, every smooth projective unitary representation of $\Diff_{\ex}(M,\mu)$ gives rise to a continuous central extension of $\fX_{\ex}(M,\mu)$ by \cite{JanssensNeeb2019}, so determining these extensions is an important first step towards a classification of projective unitary representations.

As an application of our main result, we show that (a slight adaptation of) the construction in \cite{janssensHowActionThat2024}, inspired by Ismagilov's construction~\cite{Ismagilov1980}, yields a universal central extension of $\Diff_{\ex}(M,\mu)$ in the case that $M$ is a compact, 3-dimensional manifold.

This article is structured as follows:
\begin{itemize}
    \item In Section~\ref{s:prem} we introduce the Leibniz algebra $\Omega^{n-2}(M)$. We explain the  
    relation between the Lie algebras $\mathfrak X(M,\mu)$,  $\mathfrak X_{\ex}(M,\mu)$, and
    $\oline\Omega{}^{n-2}(M)$, as well as the relations between their compactly supported versions.
    \item In Section~\ref{sec:perf} we show that the Leibniz algebra $\Omega^{n-2}(M)$ is perfect, and that its ideal of squares (i.e. its non-skew-symmetric part) 
    is the space of exact forms.
    \item In Section~\ref{sec:closed}, we use the results on $\Omega^{n-2}(M)$ to show that $\oline{\Omega}{}^{n-2}(M)$ has trivial second continuous Chevalley-Eilenberg cohomology. 
    \item In Section~\ref{sec:results}, we conclude that $\oline{\Omega}{}^{n-2}(M)$ is the universal central extension of  $\mathfrak X_{\ex}(M,\mu)$. 
    In fact, we derive this from a slightly stronger version of this result in the compactly supported setting. We also derive the continuous second Lie algebra cohomology of the Lie algebra $\fX(M,\mu)$ of all divergence-free vector fields, and of its compactly supported version $\fX_{c}(M,\mu)$.
    \item In Section~\ref{sec:LieGroups}, we construct the universal central extension of the Fréchet Lie group $\Diff_{\ex}(M,\mu)$ of exact volume-preserving 
    diffeomorphisms the case that $M$ is a three-dimensional compact manifold.
\end{itemize}
Finally, we collect a number of auxiliary results of independent interest in three appendices. Appendix~\ref{app:multivfcartan} collects some useful formulas for the Cartan calculus on multivector fields in the presence of a volume form. 
In Appendix~\ref{sec:appendixParameterPoincare} we prove a parameterised version of the compactly supported Poincar\'e Lemma. Finally, in Appendix~\ref{Appendix:PoincareLemma}, we establish a Poincar\'e Lemma for differential operators.

\paragraph*{Acknowledgements}

The authors would like to thank Claude Roger for valuable insights into the history of the problem and Peter Kristel, Karl-Hermann Neeb and Milan Niestijl for interesting discussions related to the project. The authors also thank the anonymous referee for their careful reading of the manuscript.
Part of the work was carried out during a research stay of the authors at the Erwin Schrödinger International Institute
for Mathematics and Physics, specifically during the program Higher Structures and Field-Theory. L.R. acknowledges the support of the INSMI PEPS JCJC 2024, the Accueil EC 2023-2024 of the Conseil Académique de l’UCBL, and the DFG-project Higher Lie theory.
C.V. was supported by the Romanian Ministry of Education and Research,
CNCS-UEFISCDI: project number PN-III-P4-ID-PCE-2020-2888 within
PNCDI III. B.J. was supported by the NWO grant 639.032.734 `Cohomology and representation theory of infinite dimensional Lie groups'.

\section{\texorpdfstring{The topological Leibniz algebra $\Omega^{n-2}(M)$}{The topological Leibniz algebra of forms}}\label{s:prem}

Let $M$ be a $n$-dimensional smooth manifold and $\mu$ a volume form. 
A vector field $X$ is called \emph{divergence-free} (or unimodular or volume-preserving) if $L_{X}\mu = 0$.
Since $L_X\mu=d\iota_X\mu+\iota_X d\mu=d\iota_X\mu$, this happens if and only
if $\iota_{X}\mu$ is closed. The vector field $X$ is \emph{exact divergence-free} if $\iota_{X}\mu$ is exact, that is, if it admits a \emph{potential} $\alpha \in \Omega^{n-2}(M)$ 
such that $\iota_{X}\mu = d\alpha$. 
We denote by 
$X_{\alpha}$ the unique vector field with potential $\alpha$, and the Lie algebra of exact divergence-free vector fields 
is denoted by
$
\fX_{\ex}(M,\mu) 
$.

\subsection{The non-compactly supported case}
We denote the closed forms by $\Omega^{n-2}_{\cl}(M)$. Consider the exact sequence 
\begin{equation}\label{ExactSequenceLeibniz}
\Omega^{n-2}_{\cl}(M)\longrightarrow \Omega^{n-2}(M)\stackrel{q}{\longrightarrow} \fX_{\ex}(M,\mu),
\end{equation}

where the first map is inclusion, and the second map $q \colon \Omega^{n-2}(M) \rightarrow \fX_{\ex}(M,\mu)$ 
takes $\alpha\in \Omega^{n-2}(M)$ to $X_{\alpha}$.
The bracket $[\alpha, \beta] := L_{X_{\alpha}}\beta$ on  $\Omega^{n-2}(M)$ is not skew-symmetric, however it turns $\Omega^{n-2}(M)$ into a left Leibniz algebra.

\begin{Definition}
    A left Leibniz algebra is a couple $(L,[\cdot,\cdot])$, where $L$ is a vector space and $[\cdot,\cdot]:L\times L\to L$ is a bilinear map satisfying the left Leibniz identity for all $\alpha,\beta,\gamma\in L$:
    \begin{align}\label{eq:leftleib}
        [\alpha, [\beta,\gamma]] = [[\alpha,\beta],\gamma] + 
[\beta, [\alpha,\gamma]].
    \end{align}
\end{Definition} 

Note that Leibniz algebras come in two flavours: The left Leibniz identity means that $[\alpha,\,\cdot\,]$ is a derivation, while the right Leibniz identity would mean that $[\,\cdot\,,\alpha]$ is a derivation. For Lie algebras both are equivalent and coincide with the usual Jacobi identity. We refer to \cite{FeldvossWagemann21} for an introduction to Leibniz algebras and their cohomologies.

\begin{Proposition}
The sequence \eqref{ExactSequenceLeibniz}
is a left central extension of left Leibniz algebras.
\end{Proposition}
\begin{proof}
To see that $q$ covers the ordinary Lie bracket on vector fields, note that \[d(L_{X_\alpha}\beta) = d\iota_{X_{\alpha}}d\beta = d\iota_{X_{\alpha}}\iota_{X_{\beta}}\mu = L_{X_{\alpha}}\iota_{X_{\beta}}\mu = \iota_{[X_{\alpha},X_{\beta}]}\mu.\]
The left Leibniz identity \eqref{eq:leftleib} is also an immediate calculation,
\[
L_{X_{\alpha}}L_{X_{\beta}}\gamma = L_{[X_{\alpha},X_{\beta}]}\gamma + L_{X_{\beta}}L_{X_{\alpha}}\gamma.
\]
An element $\alpha \in \Omega^{n-2}(M)$
is left central if $[\alpha, \beta] = 0$ for 
all $\beta \in \Omega^{n-2}(M)$, which is the case if and only if $X_{\alpha} = 0$. Indeed, at every point $p$ there exist forms $\beta_1,...,\beta_n$ which are 0 at $p$, whose exterior derivatives at $p$ form a basis of $\Lambda^{n-1}T^*_pM$. So $(L_X\beta_{i})_p=(\iota_Xd\beta_{i})_p=0$ for all $i$ implies $X_p=0$, and 
the kernel $\Omega^{n-2}_{\cl}(M)$ of $q$ is precisely the left centre of $\Omega^{n-2}(M)$.
\end{proof}

\begin{Remark} This extension of exact divergence-free vector fields into a Leibniz algebra is mentioned in 
\cite{Roger1995}, which quotes personal communication with Loday, who in turn attributes it to Brylinski.
\end{Remark}

Since $[\alpha, \beta] + [\beta, \alpha] = d(\iota_{X_{\alpha}}\beta + \iota_{X_{\beta}}\alpha)$ is exact, 
the Leibniz algebra structure on ${\Omega^{n-2}(M)}$ induces a 
Lie algebra structure on the quotient
\[ \oline{\Omega}{}^{n-2}(M) := {\Omega^{n-2}(M)}/{d\Omega^{n-3}(M)},\]
yielding a central extension of Lie algebras
\begin{equation}\label{hof}
H^{n-2}_{\dR}(M) \rightarrow \oline{\Omega}{}^{n-2}(M) 
\rightarrow \fX_{\ex}(M,\mu).
\end{equation}

\subsection{The compactly supported case}
There is an analogous construction for the Leibniz algebra in the compactly supported setting.
We denote by $\ceX(M,\mu)$
the Lie algebra of those compactly supported vector fields that 
admit a \emph{compactly supported} potential, 
\[
	\ceX(M, \mu) := \{X \in \fX(M,\mu) \,;\, \iota_{X}\mu = d\alpha \text{ for some } \alpha \in \Omega_{c}^{n-2}(M)\}.
\]
This is not to be confused with the Lie algebra $\fX_{c, \ex}(M,\mu)$ of \emph{all} compactly supported divergence-free vector fields, which has
$\ceX(M,\mu)$ as an ideal with abelian quotient.
The compactly supported analogue of \eqref{ExactSequenceLeibniz} is the extension 
\begin{equation}\label{qq}
\Omega^{n-2}_{c,\cl}(M)\longrightarrow \Omega^{n-2}_c(M)\stackrel{q}{\longrightarrow} \ceX(M,\mu)
\end{equation}
of the Lie algebra $\ceX(M,\mu)$
by the Leibniz algebra $\Omega^{n-2}_c(M)$. 

Setting
\[ \oline{\Omega}{}^{n-2}_c(M) := {\Omega^{n-2}_c(M)}/{d\Omega^{n-3}_c(M)},\]
we obtain a central extension  of Lie algebras
\begin{equation}\label{eq:CptSupportSequence}
H^{n-2}_{c, \dR}(M) \rightarrow \oline{\Omega}{}_{c}^{n-2}(M)
\rightarrow \ceX(M,\mu).
\end{equation}

\subsection{Topological Lie and Leibniz algebras}
We aim to show that the central extensions \eqref{hof} and \eqref{eq:CptSupportSequence} are universal in a topological setting.

\begin{Definition}
A \emph{topological Lie algebra} is a (Hausdorff) 
topological vector space $\mathfrak{g}$, together with a continuous Lie  bracket
$[\,\cdot\,,\,\cdot\,]\colon \mathfrak{g} \times \mathfrak{g} \rightarrow \mathfrak{g}$.
Similarly, a \emph{topological Leibniz algebra} is a (Hausdorff) topological vector space $L$ with a continuous Leibniz 
bracket $[\,\cdot\,,\,\cdot\,]\colon L \times L \rightarrow L$.
\end{Definition}

The Lie algebras of smooth vector fields, and of (exact) divergence-free vector fields, on a compact manifold $M$
are topological Lie algebras for the Fr\'echet topology of uniform convergence in all derivatives.
For non-compact $M$ we equip these Lie algebras with the Fr\'echet topology that comes from the inverse limit over the compact subsets $K\subseteq M$, 
and their compactly supported versions with the LF-topology that comes from the (strict) direct limit \cite[Section I.13]{Treves1967}. 

Similarly, $\Omega^{n-2}(M)$ is a topological Leibniz algebra for the inverse limit Fr\'echet topology, and $\Omega_{c}^{n-2}(M)$ for the 
direct limit LF topology.
We show that $d\Omega^{n-3}(M)\subseteq \Omega^{n-2}(M)$ and 
$d\Omega^{n-3}_{c}(M)\subseteq \Omega_{c}^{n-2}(M)$ are closed, 
making their respective quotients $\oline{\Omega}{}^{n-2}(M)$ and $\oline{\Omega}{}_{c}^{n-2}(M)$ into Hausdorff topological Lie algebras
(cf.~\cite[Theorem 1.41]{rudinFunctionalAnalysis1991}). 

\begin{Lemma}\label{exareclosed}
Let $M$ be a (not necessarily compact) orientable manifold of dimension $n$. Then for all $0\leq k\leq n-1$, the subspaces
$d\Omega^{k}(M)\subset \Omega^{k+1}(M)$ and 
$d\Omega^{k}_c(M)\subset \Omega^{k+1}_c(M)$ are closed for the Fr\'echet and LF-topology, respectively. 
\end{Lemma}
\begin{proof}
By the de Rham isomorphism, $\alpha \in \Omega^{k+1}(M)$ is exact if and only if it integrates to zero against all closed cycles,
so $d\Omega^k(M)$ is closed in $\Omega^{k+1}(M)$ for the Fr\'echet topology (cf.\ \cite[Prop.\ 5.2]{DiezJanssensNeebVizman2021}).
The inclusion $\Omega^{k+1}_c(M)\to \Omega^{k+1}(M)$ is continuous (with the LF-topolgy on the left and the Fréchet topology on the right), hence $d\Omega^{k}(M)\cap \Omega^{k+1}_c(M)$ is closed in $ \Omega^{k+1}_c(M)$. We now consider for any $\alpha\in \Omega_{\cl}^{n-k-1}(M)$ the following functional:
$$
F_\alpha:\Omega^{k+1}_c(M)\to \mathbb R,~~~ \beta\mapsto F_\alpha(\beta)=\int_{M}\alpha\wedge \beta
$$
These functionals are continuous in the LF-topolgy, since they are continuous when restricted to spaces of $\beta$'s with any fixed compact support. This means that $\ker(F_\alpha)$ and hence $\bigcap_\alpha \ker(F_\alpha)$ are closed. Hence also 
 $$ C=d\Omega^{k}(M)\cap \Omega^{k+1}_c(M)\cap \bigcap_{\alpha\in \Omega_{\cl}^{n-k-1}(M)} \ker(F_\alpha)$$
is closed. We claim this space is equal to $d\Omega^{k}_c(M)$. 
To show that $d\Omega^{k}_{c}(M) \subseteq C$, note that $\beta\in d\Omega^{k}_c(M)$ implies that $F_\alpha(\beta)=0$ for all closed $\alpha$.
For the converse, recall that by Poincaré duality, the pairing 
$H^{n-k-1}_{\dR}(M)\times H^{k+1}_{\dR, c}(M)\to \mathbb R$ defined by $([\alpha],[\beta])\mapsto F_\alpha(\beta)$ is non-degenerate in both entries. Since any $\beta\in C$ has compact support and is closed, it defines a class in $H^{k+1}_c(M)$. By Poincaré duality, being in the kernel of all $F_\alpha$ means that this class is zero, so $\beta\in d\Omega^{k}_c(M)$ and  $C \subseteq d\Omega^{k}_c(M)$. So $d\Omega^{k}_c(M) = C$ is closed, and $\Omega^{k+1}_{c}(M)/d\Omega^{k}_{c}(M)$ is a Hausdorff topological vector space.
\end{proof}

Equipped with this topology, the exact sequence \eqref{hof} is a central extension of Fr\'echet--Lie algebras, and 
\eqref{eq:CptSupportSequence} is a central extension of LF--Lie algebras.

\begin{Remark}
In this paper 
we will generally restrict attention to the case $\mathrm{dim}(M)\geq 3$, because the case $\mathrm{dim}(M)=2$ requires a different approach. In the two-dimensional case, a volume form is the same as a symplectic form, $\fX_{\ex}(M,\mu)$ is the Lie algebra of Hamiltonian vector fields, and $\oline{\Omega}{}^{n-2}(M) = C^{\infty}(M)$ is the Poisson algebra. In general the Poisson algebra is neither perfect nor centrally closed, and its central extensions were investigated elsewhere \cite{JV16}.
\end{Remark}

\section{Perfectness and the ideal of squares}\label{sec:perf}

The goal of this section is to show that the Leibniz algebras $\Omega^{n-2}(M)$ and $\Omega_{c}^{n-2}(M)$ are perfect, 
and that their ideal of squares is given by the exact forms.

\subsection{Perfectness}

We will work in local coordinates in which $\mu$ takes the standard form. 
\begin{Lemma}\label{Lemma:localcoordinates} Let $\mu$ be a volume form on a manifold $M$. Let $U \subseteq M$ be a coordinate neighbourhood that is diffeomorphic to $\R^n$. Then $U$ admits coordinates in which $\mu$ takes the standard form   $\mu = dx_1\wedge \cdots \wedge dx_n$.
\end{Lemma}
\begin{proof}

If $\mu = f dx'_1\wedge \ldots \wedge dx'_n$ in local coordinates $x'_i \in \R^n$, then set $x_i := x'_i$ for $i\geq 2$, and replace $x'_1$ by
$x_{1}(x'_1, \ldots, x'_n):= \int_{0}^{x'_1}f(s, x'_2, \ldots, x'_n)ds$.
\end{proof}

Since $\mu$ is non-degenerate, it induces an isomorphism $A \mapsto i_{A}\mu$ between multivector fields and differential forms. In particular, we can rephrase 
the Cartan calculus (contraction, Lie derivative, de Rham differential) in terms of multivector fields. We will use this perspective in the sequel, since it makes certain calculations more traceable. A summary of the most relevant formulas for us from this perspective can be found in Appendix~\ref{app:multivfcartan}.

\begin{Lemma}\label{Lemma:localperfect}
Let $U\subset \mathbb R^n$ with $n\geq 3$ be open and connected, and let $\mu = dx_1\wedge \cdots \wedge dx_n$.
Then any element in  $\Omega_{c}^{n-2}(U)$ can be expressed by at most $\binom{n}{2}\cdot (n+1)$ commutators.
\end{Lemma}

\begin{proof}
We show this step by step, using the bivector field expressions from Appendix \ref{app:multivfcartan} instead of $(n{-}2)$-forms:
\begin{enumerate}
    \item Let $x,y,z$ be coordinates among $x_1,...,x_n$ and let $h\in C^{\infty}_c(U)$. We realise $\partial_xh\partial_y\wedge\partial_z$ as 
    a single commutator  in  $\mathfrak X^2_{c}(U)\cong\Omega_{c}^{n-2}(U)$. For that we choose a function $\phi$ which is compactly supported in $U$ and coincides with $y$ on the support of $h$. Then 
    $[\phi \partial_x\wedge \partial_y, h\partial_y\wedge\partial_z]=L_{\partial_x}(h\partial_y\wedge\partial_z)=\partial_xh\partial_y\wedge\partial_z$. Note that this works even when $x=z$.
    \item Suppose $g$ satisfies $\int_U gdx_1\wedge ...\wedge dx_n=0$. Then the class ${[gdx_1\wedge ...\wedge dx_n]}$ is zero in the compactly supported cohomology of $U$, hence $g=\sum_{i=1}^n\partial_ih_i$ for $h_i\in C^{\infty}_c(U)$. In particular $g\partial_x\wedge\partial_y$ can be written as a sum of $n$ commutators.
    \item Let $x, y, z$ be three coordinates among $x_1, \ldots, x_n$. Let $f \in C^{\infty}_{c}(U)$, 
and choose $\phi \in C^{\infty}_{c}(U)$ such that $\phi$ agrees with $yz$ on $\mathrm{supp}(f)$.
With $A = \phi \partial_{x} \wedge \partial_y$ and $B = f\partial_{y}\wedge \partial_z$, we have $\delta(A)=z\partial_x$
on $\supp(B)$ by
Corollary~\ref{Corollary:deltaBivector}, so $[A,B] = L_{z\partial_x}(f\partial_y\wedge \partial_z) = f\partial_x \wedge \partial_y + z f_x \partial_y \wedge \partial_z$ by Proposition~\ref{Prop:bracketbivectors}.
Since $zf_x$ integrates to zero over $U \subseteq \R^n$, this means that $ f\partial_x \wedge \partial_y$ can be written as a sum of $n+1$ commutators.
\end{enumerate}
In total this means that a form supported on $U$ can be written as the sum of $\binom{n}{2}\cdot (n+1)$ commutators.
\end{proof}

\begin{Remark}
Note that the statement is false for $n=2$. In this case $\Omega^{n-2}_{c}(V) = C^{\infty}_{c}(V)$ is the compactly supported Poisson algebra, whose commutator ideal consists of functions that integrate to zero \cite[Section 12]{ALDM74}, \cite[Prop.~3.1]{JV16}.  
\end{Remark}

We can now prove the global statement for manifolds of dimension $\mathrm{dim}(M)\geq 3$:

\begin{Theorem}\label{Thm:perfect}
The Leibniz algebras $\Omega^{n-2}_{c}(M)$ and $\Omega^{n-2}(M)$ are perfect.
\end{Theorem}
\begin{proof}
The fact that $\Omega_{c}^{n-2}(M)$ is perfect follows from Lemma~\ref{Lemma:localperfect} by a partition of unity argument. The statement for $\Omega^{n-2}(M)$ needs a slightly refined argument.

Let $\mathcal U$ be a covering of $M$ by relatively compact open sets in which $\mu$ takes the standard form $\mu = {dx_1\wedge \cdots \wedge dx_n}$. Since the covering dimension of an $n$-dimensional manifold is $n$, Ostrand's theorem (\cite[Lemma 3]{ostrandCoveringDimensionGeneral1971}, refining the 
Brouwer-Lebesgue Paving Principle, cf.\ e.g.\ \cite{HurewiczWallman1941}) states that there exist open sets $V_{i,k}$, $i\in\{1,...,n+1\}$, $k\in\mathbb N$ with the following properties:
\begin{itemize}
	\item Each $V_{i,k}$ is a connected open subset of an element in $\mathcal U$.
	\item For fixed $i$ and $k\neq l$, $V_{i,k}\cap V_{i,l}=\emptyset$
	\item $V_{i,k}$ cover $M$.
		\end{itemize}
In particular $W_i=\bigsqcup_k V_{i,k}$ gives an open cover of $M$ by $n+1$ sets. We can now pick a partition of unity $\rho$ with respect to $\{W_i\}$. Let  $\alpha\in \Omega^{n-2}(M)$. Since the restriction $\rho_i\alpha|_{V_{i,k}}$ is compactly supported, it can be written as a sum of $\binom{n}{2}(n+1)$ commutators 
in $\Omega^{n-2}_{c}(V_{i,k})$.
Since the $V_{i,k}$ are disjoint for fixed $i$, these commutators can be assembled and we obtain an expression of $\rho_i\alpha$ in terms of $\binom{n}{2}(n+1)$ commutators in $\Omega^{n-2}(W_i)$. But this means that $\alpha=\sum_i \rho_i\alpha$ is a sum of at most $\binom{n}{2}(n+1)^2$ commutators
in $\Omega^{n-2}(M)$.
\end{proof}

\begin{Corollary}[\cite{Omori1974, Lichnerowicz1974}]\label{PerfectLA}
The Lie algebras $\fX_{\ex}(M, \mu)$ and $\ceX(M, \mu)$ are perfect for $\mathrm{dim}(M)\geq 3$.
\end{Corollary}

\begin{Remark}
Note that Corollary \ref{PerfectLA} is stated in \cite{Lichnerowicz1974} and 
\cite[Chapter X.3]{Omori1974}, and that in the compact case
Theorem~\ref{Thm:perfect} would follow from Corollary~\ref{PerfectLA} together with
the characterization of the ideal of squares that we prove in
Theorem~\ref{Thm:squares} below. Unfortunately there seems to be an error in the proof of \cite{Omori1974},
and a small gap in the proof of \cite{Lichnerowicz1974}, which is why we provide this independent proof inspired by the approach used in \cite{Omori1974}. 
\end{Remark}

\subsection{The ideal of squares}
In this subsection we will investigate the \emph{ideal of squares} $\Omega^{n-2}(M)^{\sq}$, i.e. the left ideal generated by $[\alpha,\alpha]$ for $\alpha\in \Omega^{n-2}(M)$. Let us start by observing that for any $\alpha\in \Omega^{n-2}$, $[\alpha,\alpha]=L_{X_{\alpha}}\alpha=d\iota_{X_\alpha}\alpha$. In particular the ideal of squares is contained in the exact forms. The goal of this section is to prove the converse, i.e.\ that any exact form can be written as a sum of squares. We start with the compactly supported statement for cubes in $\mathbb R^n$, then we prove the general local statement and then the global version. 

\begin{Lemma} \label{lem:cubesquares}
Let $U\subset \mathbb R^n$ be a relatively compact cube (i.e. the cartesian product of open finite intervals) with the canonical volume form $\mu$ and $\beta\in \Omega^{n-3}_c(U)$.
Then $\beta$ can be written as a sum $\beta = \sum_{i}\iota_{X_{\alpha^i}}\alpha^i$ 
of at most $4 \binom{n}{3}$ forms $\alpha^i$ with compact supports in $U$.  
\end{Lemma}

Again, in the proof we will employ multivector fields in the calculations. It will be useful to know that for a form $\alpha={\iota_{X\wedge Y}\mu}$, the formula \eqref{eq:xalp} in Appendix \ref{app:multivfcartan} implies:
\begin{align}\label{eq:contractiondecomposable}
    \iota_{X_\alpha}\alpha=-\iota_{[X,Y]\wedge X\wedge Y}\mu.
\end{align}

\begin{proof}
Since every $\beta \in \Omega^{n-3}_{c}(U)$ can be written as a sum of at most $\binom{n}{3}$ terms of the form 
$$\beta=g\iota_{\partial_{x}\wedge\partial_{y}\wedge\partial_{z}}\mu$$ for some coordinates $x,y,z$ among $x_1,...,x_n$,
it suffices to prove that each such term can be written as a sum of at most 4 terms of the form $\iota_{X_{\alpha^i}}\alpha^i$.
\begin{enumerate}
    \item We first show the statement for the case where $g$ is a total divergence (as a function of $x,y,z$ the other variables being treated as parameters). Let $g= \partial_xf^1 + \partial_y f^2 +\partial_zf^3$ with compactly supported $f^i$.
   Set $X= \phi\partial_x$ for a compactly supported function $\phi \colon \R^n \rightarrow \R$ that is $1$ on the support of $f^1$, and take $Y=\partial_y-f^1\partial_z$. Then $\alpha^1=\iota_{X\wedge Y}\mu$ is 
   compactly supported, and by formula \eqref{eq:contractiondecomposable} we have $\iota_{X_{\alpha^1}}\alpha^1=\partial_xf^1\iota_{\partial_{x}\wedge\partial_{y}\wedge\partial_{z}}\mu$. Similarly we can find $\alpha^2, \alpha^3$ such that 
   $$
   \beta= \iota_{X_{\alpha^1}}\alpha^1+\iota_{X_{\alpha^2}}\alpha^2+\iota_{X_{\alpha^3}}\alpha^3
   $$

\item Consider now $X=h\partial_x$ and $Y=\partial_y-f\partial_z$ for compactly supported $f$ and $h$. Then $[X,Y]\wedge X\wedge Y=h^2\partial_xf{\partial_{x}\wedge\partial_{y}\wedge\partial_{z}}$, hence $\beta = h^2\partial_xf\iota_{\partial_{x}\wedge\partial_{y}\wedge\partial_{z}}\mu$ can be realized as 
\[
\beta = \iota_{X_\alpha}\alpha.
\]
   
   \item  
   Let $U=V\times W$ for a cube $V\subset \mathbb R^3$ and $W\subset \mathbb R^{n-3}$. Let now $g$ be a function with compact support contained in  a cube $C \times D \Subset V\times W$ (Here $\Subset$ denotes relative compactness, i.e. the closure of $C\times D$ is compact in $
   V\times W$), and let $C' \Subset V$ be a larger cube, (i.e. $C \Subset C'$). The function $H=\iiint_{V} gdxdydz$ only depends on the other $n-3$ coordinates and has support in $D$. We exhibit compactly supported functions $f, h$ such that \begin{equation}\label{eq:goededivergentie}
   \iiint_{V}h^2\partial_xf dxdydz = H.
   \end{equation}
   Let $f = x H \chi$ for a compactly supported function $\chi \colon V \rightarrow \R$ which is $1$ on $C'$. (Then $\partial_x f = H$ on $C' \times W$.) Let $h = \phi \psi$ be the product of compactly supported functions $\phi \colon V \rightarrow \R$  and $\psi \colon W \rightarrow \R$ such that $\psi$ is $1$ on $D$ and $\phi$ is constant on $C$, zero outside $C'$, and it satisfies $\iiint_{V}\phi^2 dxdydz = 1$.

    Then 
    \begin{eqnarray*}
    \iiint_{V} h^2 \partial_xf dxdydz &=& 
    H\psi^2 \iiint_{V} \phi^2 \partial_x(x\chi) dxdydz\\
    &=& H\psi^2 \iiint_{V} \phi^2 dxdydz
    \end{eqnarray*}
    because $\chi$ is $1$ whenever $\phi$ is nonzero.
    Since $\psi$ is $1$ whenever $H$ is nonzero, equation 
    \eqref{eq:goededivergentie} follows.

Now the compactly supported function $g-h^2\partial_xf$ integrates to zero over $V$, hence it is a total divergence.
By the parametrized Poincaré Lemma (\ref{lem:PoincParam}), there exist compactly supported smooth functions $f^1,f^2,f^3$ such that
$$ 
g=h^2\partial_xf + \partial_xf^1+\partial_yf^2+\partial_zf^3
$$
Hence $\beta=g\iota_{\partial_{x}\wedge\partial_{y}\wedge\partial_{z}}\mu$ can be realized as the sum of four terms of the type $\iota_{X_\alpha}\alpha$.
\end{enumerate}
\end{proof}

The problem with the above Lemma is that it only works on cubes in $\mathbb R^n$. However, because we want to apply the same technique as in Theorem \ref{Thm:perfect} to globalize the construction, we need to have the statement for any connected subset of $\mathbb R^n$.

\begin{Lemma} \label{lem:localsquares} Let $V\subset \mathbb R^n$ with the canonical volume form $\mu$ and $\beta\in \Omega^{n-3}_c(V)$. Then $\beta$ can be written as a sum $\beta = \sum_{i}\iota_{X_{\alpha^i}}\alpha^i$  of at most $4 \binom{n}{3}$ forms $\alpha^i$ with compact supports in $V$.  
\end{Lemma}
\begin{proof}
Let $U$ be a precompact cube containing $V$. By Lemma \ref{lem:cubesquares} we can find $4 \binom{n}{3}$ forms $\tilde \alpha_i\in \Omega^{n-2}_c(U)$ such that $\beta = \sum_i\iota_{X_{\tilde \alpha^i}}\tilde \alpha^i$. We note that by construction each of the $\tilde \alpha^i$ is the contraction of two vector fields into $\mu$, i.e.  $\tilde\alpha^i=\iota_{X^i\wedge Y^i}\mu$ for some vector fields $X^i,Y^i\in \mathfrak X(U)$. Let $\chi\in C^{\infty}_c(V)$ be a function such that $\chi|_{\supp(\beta)}=1$. We set $\alpha^i=\chi\cdot \tilde\alpha^i=\iota_{X^i \wedge \chi Y^i}\mu$. These clearly have support in $V$. Moreover, they satisfy:
$$\iota_{X_{\alpha^i}}\alpha^i=-\iota_{[\chi X^i,Y^i]\wedge \chi X^i\wedge Y^i}\mu = \chi^2\iota_{X_{\tilde \alpha^i}}\tilde\alpha^i.$$
This means:
$$
\sum_i \iota_{X_{\alpha^i}}\alpha^i= \sum_i\chi^2\iota_{X_{\tilde \alpha^i}}\tilde\alpha^i=\chi^2\beta=\beta.
$$
\end{proof}

We can now prove the global statement:

\begin{Proposition}
    Let $M$ be a manifold of dimension $n$. Then every 
$\beta \in \Omega^{n-3}(M)$ can be written as a sum of at most $4(n+1)\binom{n}{3}$ terms of the form $\iota_{X_{\alpha^i}}\alpha^i$ for $\alpha^i \in \Omega^{n-2}(M)$. If $\beta$ is compactly supported, then each $\alpha^{i}$ can be chosen to be compactly supported as well.
\end{Proposition}
\begin{proof}
As in the proof of Theorem \ref{Thm:perfect}, we pick a covering $\mathcal U$ by (cube-shaped) charts on which $\mu$ has canonical form. We again apply Ostrand's theorem \cite[Lemma 3]{ostrandCoveringDimensionGeneral1971} to obtain a finite covering $W_i$ ($i\in \{1,...,n+1\}$) each of which is a countable disjoint union $W_i=\bigsqcup_k V_{i,k}$ of connected subsets of elements of $\mathcal U$.

Using a partition of unity, we can write 
$\beta = \sum_{i=1}^{n+1}\sum_{k=1}^{\infty}\beta_{i,k}$ with $\supp(\beta_{i,k})\subseteq V_{i,k}$. By Lemma~\ref{lem:localsquares}, every 
(compactly supported) $\beta_{i,k} \in \Omega^{n-3}_{c}(V_{i,k})$ can be written as a finite sum 
\[\beta_{i,k} = \sum_{j=1}^{4{\binom{n}{3}}} \iota_{X_{\alpha_{i,k}^{j}}} \alpha_{i,k}^{j}\]
for some compactly supported $\alpha_{i,k}^{j} \in \Omega^{n-2}_{c}(V_{i,k})$.
Note that the infinite sum
\[\alpha_{i}^{j} := \sum_{k=1}^{\infty} \alpha_{i,k}^{j}\]
is well defined because the $V_{i,k}$
are mutually disjoint for fixed $i$.

If we set $\beta_i := \sum_{k=1}^{\infty} \beta_{i,k}$, then 
\begin{eqnarray*}
\sum_{j=1}^{4{\binom{n}{3}}} \iota_{X_{\alpha^{j}_{i}}}\alpha^{j}_{i}
=
\sum_{j=1}^{4{\binom{n}{3}}} 
    \sum_{k=1}^{\infty} \iota_{X_{\alpha^{j}_{i,k}}}
    \sum_{l=1}^{\infty} \alpha^{j}_{i,l}
=
\sum_{j=1}^{4{\binom{n}{3}}} 
    \sum_{k=1}^{\infty} \iota_{X_{\alpha^{j}_{i,k}}} \alpha^{j}_{i,k}
    =
    \sum_{k=1}^{\infty} 
    \sum_{j=1}^{4{\binom{n}{3}}} \iota_{X_{\alpha^{j}_{i,k}}} \alpha^{j}_{i,k}
    =
    \sum_{k=1}^{\infty}\beta_{i,k} = \beta_i.
\end{eqnarray*}
Since $\beta = \sum_{i=1}^{n+1} \beta_{i}$, this concludes the non-compactly supported case. If $\beta$ is compactly supported, one can arrange that only a finite number of the $\beta_{i,k}$ are non-zero. Then we can arrange that only finitely many $\alpha_{i,k}^j$ are non-zero, such that $\alpha_i^j$ are compactly supported. 
\end{proof}

Since $d\beta = \sum_{i}d \iota_{X_{\alpha^{i}}}\alpha^i = \sum_{i}[\alpha^i, \alpha^i]$, we have the following theorem as a direct consequence:
\begin{Theorem}\label{Thm:squares}
The ideal of squares in $\Omega^{n-2}(M)$ is equal to the space of exact forms $ d\Omega^{n-3}(M)$ and the ideal of squares in $\Omega^{n-2}_c(M)$ is equal to $ d\Omega^{n-3}_c(M)$.     
\end{Theorem}

\section{\texorpdfstring{The Lie algebra ${\Omega^{n-2}(M)}/{d\Omega^{n-3}(M)}$ is centrally closed}{The quotient Lie algebra is centrally closed}}\label{sec:closed}

We recall that for brevity we denote $\Omega^{k}(M)/d\Omega^{k-1}(M)$ by $\oline{\Omega}{}^{k}(M)$.

The goal of this section is to show that the second jointly continuous Lie algebra cohomology of $\oline{\Omega}{}^{n-2}(M)$ vanishes. 
Since $\oline{\Omega}{}^{n-2}(M)$ is a central extension of the perfect Lie algebra $\mathfrak X_{\ex}(M,\mu)$ (see \eqref{hof}), this will imply that the former is the universal central extension of the latter.

\subsection{Continuous cohomology for Lie and Leibniz algebras}\label{3.1}

We recall the Chevalley-Eilenberg Lie algebra cohomology in the continuous setting, as well as the corresponding Leibniz algebra cohomology.

\paragraph{Lie algebra cohomology.}

Let $\fg$ be a topological Lie algebra. Then a  $\mathfrak g$-module is a topological vector space 
$\mathfrak M$, together with a left $\fg$-action $(x, m) \mapsto x\cdot m$ that is continuous as a bilinear map $\mathfrak g\times \mathfrak M\to \mathfrak M$. 
The Chevalley-Eilenberg differential on the complex 
$C^n(\mathfrak{g},\mathfrak M)$ of continuous alternating $n$-linear maps $\psi \colon \fg^n \rightarrow \mathfrak M$  is given by 
\begin{align}\label{CE}
\dd \psi(x_1,\ldots,x_{n+1}):= &\sum_{ i<j}
(-1)^{i+j} \psi([x_i,x_j],x_1,\ldots,\widehat{x}_i, \ldots, 
\widehat{x}_j, \ldots, x_{n+1})\\ + &\sum _i (-1)^{i+1} {x_i}\cdot\psi(x_1,\ldots,\widehat{x}_i, \ldots, x_{n+1}),\nonumber
\end{align}
so in particular $\dd \psi(x)=x\cdot \psi$ for $\psi \in C^{0}(\fg, \mathfrak M)\simeq \mathfrak{M}$.
The cohomology of this complex, denoted  by $H^n(\mathfrak{g},\mathfrak M)$,
is called the \emph{continuous Lie algebra cohomology} of the locally convex Lie algebra $\mathfrak{g}$. 
In the same vein, we denote by $H_{\mathrm{alg}}^n(\mathfrak{g},\mathfrak M)$ the cohomology of the complex 
$C_{\mathrm{alg}}^n(\fg,\mathfrak{M})$ of alternating linear maps without continuity assumptions.

For the trivial representation $\mathfrak M=\mathbb R$, the second term in Equation \eqref{CE} vanishes and we obtain the continuous Lie algebra cohomology $H^n(\mathfrak{g},\R)$ with trivial coefficients. In degree 1 this cohomology $H^1(\mathfrak{g},\R)$ is the topological dual of the abelian Lie algebra 
$(\fg/\overline{[\fg,\fg]})$, where $\overline{[\fg,\fg]}$ is the closure of the commutator ideal.  In particular, a locally convex Lie algebra is 
topologically perfect ($\fg = \overline{[\fg,\fg]}$) if and only if $H^1(\fg,\R)$ vanishes.

Since the cohomology in degree 2 classifies the continuous central extensions of $\fg$ (cf. Section \ref{sec:results}), we call a locally convex Lie algebra $\mathfrak g$ \emph{centrally closed} if $H^2(\mathfrak g, \mathbb R)=0$.

\paragraph{Leibniz algebra cohomology.}
A topological left Leibniz algebra is a locally convex vector space $\mathfrak L$ with a continuous bilinear map 
$[\,\cdot\,, \,\cdot\,]:\mathfrak L\times \mathfrak L\to \mathfrak L$ such that the left Jacobi identity holds:
$$[x,[y,z]]=[[x,y],z]+[y,[x,z]].$$
There are various conventions for Leibniz cohomology of Leibniz algebras. Here we follow 
\cite{FeldvossWagemann21}.

Let $\fL$ be a topological left Leibniz algebra. A left-module for $\fL$ is a topological vector space $\mathfrak M$ with a continuous left action 
$\fL\times \mathfrak M\to \mathfrak M$ satisfying $[x,y] \cdot m=x\cdot (y \cdot m)- y\cdot (x \cdot m)$. The Loday complex is the complex $CL^n(\fL,\mathfrak M)$ of  jointly continuous $n$-linear maps $\fL^n \rightarrow \mathfrak{M}$. The differential is given by a convenient rewriting of \eqref{CE}:
\begin{align*}
\dd \psi(x_1, \ldots x_{n+1}) = &\sum_{i<j}(-1)^i\psi(x_1,\ldots,\widehat{x_i},\ldots,[x_i,x_j],\ldots,x_{n+1}),
\\ + &\sum _i (-1)^{i+1} {x_i}\cdot\psi(x_1,\ldots,\widehat{x}_i, \ldots, x_{n+1})
\end{align*}
where the term $[x_i,x_j]$ is placed in the $j$-th position.
We denote the cohomology of this complex by $HL^\bullet(\fL,\mathfrak M)$.
Similarly, we denote by $HL_{\mathrm{alg}}^\bullet(\fL,\mathfrak M)$ the cohomology of the Loday complex $CL_{\mathrm{alg}}^n(\fL,\mathfrak M)$
of multilinear maps without continuity assumptions.

Recall that the ideal of squares (also called the Leibniz kernel) $\fL^{\sq}$ is the left ideal spanned by elements of the form $[x,x]$ for $x\in \fL$. Let $\fL_{\Lie} = \fL / \oline{\fL}{}^{\sq}$ be the largest quotient of $\fL$ that is a Hausdorff locally convex Lie algebra.
The projection $\pi \colon \fL \rightarrow \fL_{\Lie}$ gives a pullback map 
$\pi^*\colon C^n(\fL_{\Lie},\R)\rightarrow CL^n(\fL,\R)$,
which is a chain map from the Chevalley-Eilenberg complex to the Loday complex.
In particular, we have maps on the level of cohomology groups
\[
\pi^* \colon H^{\bullet}(\fL_{\Lie},\R) \rightarrow HL^{\bullet}(\fL,\R).
\]

Later in the article, we will show exactness of cocycles in $\fL_{\Lie}$ by showing the corresponding exactness in $\fL$. For this to work, we will need the following result:

\begin{Proposition}\label{prop:inj} The map $\pi^* \colon H^{n}(\fL_{\Lie},\R) \rightarrow HL^{n}(\fL,\R)$ is injective for $n=1$ and $n=2$.
\end{Proposition}
\begin{proof}
The statement in degree one follows from the fact that the pullback $\pi^* \colon  C^1(\fL_{\Lie},\R) \to CL^1(\fL,\R)$ is injective, since there are no coboundaries to divide out. Here $ C^1(\fL_{\Lie},\R)$ is just the continuous dual $\fL_{\Lie}'$ of $\fL_{\Lie}$ and $CL^1(\fL,\R)$ the continuous dual $\fL'$ of $\fL$.

Let $[\psi] \in H^2(\fL_{\Lie},\R)$. If $\pi^*[\psi] = 0$, then 
$\pi^*\psi = \dd c$ for $c \in \fL'$.
So $c([\alpha,\alpha]) = \psi(\pi(\alpha), \pi(\alpha))$ which is zero 
because $\psi$ is skew-symmetric. This means that $c \colon \fL \rightarrow \R$ vanishes on $\fL^{\sq}$, and hence on $\oline{\fL}{}^{\sq}$ because $c$ is continuous. So $c$ induces a continuous map $\fL_{\Lie} = \fL/ \oline{\fL}{}^{\sq}\rightarrow \R$, which is automatically a primitive of $\psi$, hence $[\psi] = 0$ in 
$H^2(\fL_{\Lie},\R)$.
\end{proof}

\begin{Remark}
We will be mainly interested in the cases $\fL = \Omega^{n-2}_{c}(M)$ and $\fL = \Omega^{n-2}(M)$, where 
$\fL^{\sq}$ is closed by Theorem~\ref{Thm:squares} and 
Lemma \ref{exareclosed}.
\end{Remark}

Let $\fL'$ be the continuous dual of $\fL$, equipped with the coadjoint action $(x \cdot T) y := -T([x,y])$.
We close this subsection by noting that to any 2-cochain $\psi\in CL^2(\fL, \R)$, we can associate a (not necessarily continuous) 1-cochain $\widehat{\psi}\in CL_{\mathrm{alg}}^1(\fL,\fL')$ by $\widehat\psi(x)y=\psi(x,y)$. 
Similarly, a 1-cochain $\eta \in CL^1(\fL,\R)$ corresponds to a 0-cochain $\widehat{\eta} \in CL_{\mathrm{alg}}^0(\fL,\fL')$. 
Then $\widehat{\psi}$ is a cocycle if and only if $\psi$ is a cocycle, and that $\psi = \dd \eta$ if and only if 
$\widehat{\psi} = \dd \widehat{\eta}$. Since the map $CL^1(\fL,\R) \rightarrow  CL_{\mathrm{alg}}^0(\fL,\fL')$ is bijective, we have the following:

\begin{Lemma}\label{lem:2coc1coc} 
The map $HL^2(\fL,\R) \hookrightarrow HL_{\mathrm{alg}}^1(\fL,\fL')$ defined by $[\psi] \mapsto [\widehat{\psi}]$ is injective.
\end{Lemma}

This statement is an instance of a much more general phenomenon, cf.\ e.g. \cite[Corollary 1.5]{FeldvossWagemann21}.

\subsection{The perfectness trick}

Lie algebra cohomology with diagonal cocycles is extensively developed in the monograph \cite{FuksCohomologyOfInfinite1986}.
The perfectness trick refers to a reasoning that infers that every continuous 2-cocycle is diagonal 
directly from the perfectness of a Lie algebra (and the fact that the bilinear map given by its bracket is diagonal).
It was used in \cite{JV16} for the Poisson bracket on functions on a symplectic manifold
and in \cite{janssensNiestijl2024} for the Lie bracket on vector fields on an arbitrary manifold.
Here we extend it to the left Leibniz bracket on differential $(n-2)$-forms associated with a volume form.

To any continuous 2-cocycle $\psi$ on the Leibniz algebra $\Omega_c^{n-2}(M)$, one associates  
a continuous linear map with values in the continuous linear dual $\Omega_c^{n-2}(M)'$:
\begin{equation}\label{psihat}
\widehat\psi:\Omega^{n-2}_{c}(M)\to \Omega^{n-2}_{c}(M)',\quad \widehat\psi(\alpha)\beta:=\psi(\alpha,\beta).
\end{equation}
It is a 1-cocycle on the Leibniz algebra $\Omega_c^{n-2}(M)$, 
for the action $(\alpha\cdot T)(\beta)=-T([\alpha,\beta])=-T(L_{X_\alpha \beta})$ on $\Omega_c^{n-2}(M)'$ that is dual to the adjoint Leibniz algebra action, i.e.
\begin{equation}\label{eq:OneCocyclefirst}
\widehat\psi([\alpha,\beta]) = \alpha \cdot \widehat\psi(\beta) - \beta \cdot \widehat\psi(\alpha),\quad\forall\alpha,\beta\in\Omega_c^{n-2}(M).
\end{equation}
A Leibniz 2-cocycle $\psi$ on  $\Omega_c^{n-2}(M)$  is called diagonal 
if $\psi( \alpha,\beta) =0$ whenever $\supp( \alpha ) \cap\supp(\beta)=\emptyset$.
This implies that the induced 1-cocycle is support-decreasing:
 $\supp(\widehat\psi(\alpha))\subseteq \supp(\alpha)$ for every $\alpha\in\Omega_c^{n-2}(M)$.

\begin{Proposition}\label{Lemma:ClosedImpliesLocal}
Let $n\geq 3$. Then
any continuous 2-cocycle $\psi$ on the Leibniz algebra $\Omega^{n-2}_{c}(M)$ is diagonal.
Moreover, the Leibniz 1-cocycle $\widehat\psi$ in \eqref{psihat} is a distribution-valued differential operator of locally finite order on $\Omega_{c}^{n-2}(M)$.
\end{Proposition}

\begin{proof}
Let $\psi$ be a Leibniz 2-cocycle on $\Omega_{c}^{n-2}(M)$, thus 
$\psi(\alpha,[\beta,\gamma])=\psi([\alpha,\beta],\gamma)+\psi(\beta,[\alpha,\gamma])$.
For $\alpha, \beta \in \Omega^{n-2}_{c}(M)$ with $\supp(\alpha) \cap \supp(\beta) = \emptyset$, let $U \subseteq M$ be open with 
$\supp(\alpha) \subseteq U$ and $\supp(\beta) \cap U = \emptyset$. 
By Theorem \ref{Thm:perfect}, $\Omega^{n-2}_{c}(U)$ is perfect, thus we can write $\alpha = \sum_{i=1}^{N} [\alpha'_i, \alpha''_i]$ for 
$\alpha'_i, \alpha''_i \in \Omega^{n-2}_{c}(M)$ with support contained in $U$. By the above cocycle identity, we then have
\begin{equation*}
\psi(\alpha,\beta) = \sum_{i=1}^{N} \psi([\alpha'_i, \alpha''_i],\beta)
 = \sum_{i=1}^{N}  \psi(\alpha'_i,[\alpha''_i,\beta])-\psi(\alpha''_i,[\alpha'_i ,\beta]) ,
\end{equation*}
which is zero because $\alpha'_i$ and $\alpha''_i$ have support which is disjoint from $\beta$. 
Thus $\psi$ is diagonal.

The second statement is a consequence of the first, by applying to $\widehat\psi$ Peetre's Theorem, more precisely its vector-bundle version, cf.\ Theorem \ref{thm:peetre} in Appendix \ref{appendix:poincare}, which asserts that support-decreasing continuous linear maps from the compactly supported section space of a vector bundle to the continuous dual of the compactly supported section space of a vector bundle is a distribution-valued differential operator of locally finite order.
\end{proof}

Every continuous Lie algebra 2-cocycle $\phi$ on $\fX_{c}(M,\mu)$ lifts via the projection $q$ in \eqref{qq} to a Leibniz 2-cocycle $\psi$ on $\Omega^{n-2}_{c}(M)$,
which in addition is skew. 
As in the previous section we associate the continuous Leibniz 1-cocycle $\widehat\psi:\Omega^{n-2}_{c}(M)\to\Omega^{n-2}_{c}(M)'$,
which in addition vanishes on $ d\Omega_{c}^{n-3}(M)$.
By the perfectness trick in Proposition \ref{Lemma:ClosedImpliesLocal}, the Leibniz 2-cocycle $\psi$ is diagonal and 
the Leibniz 1-cocycle $\widehat\psi$ is a continuous differential operator of locally finite order on $\Omega_{c}^{n-2}(M)$.

For proving the results about central extensions of the Lie algebra of divergence-free vector fields, we need to use differential forms with polynomial coefficients,
hence without compact support.
Lemma \ref{lem:sheaf} in Appendix B allows to extend the map $\widehat\psi$ to a continuous differential operator, which we denote by the same letter $\widehat\psi$.

\begin{Lemma}\label{lemma:extendpsi}
Let $\phi\in C^2(\mathfrak X_c(M),\R)$ be a cocycle and $\psi=\pi^*\phi$ as above. Then $\widehat \psi\colon \Omega^{n-2}_c(M) \rightarrow \Omega^{n-2}_{c}(M)'$ can be uniquely extended to an operator
\begin{equation}\label{doi}
\widehat\psi \colon \Omega^{n-2}(M) \rightarrow \Omega^{n-2}_{c}(M)'
\end{equation}
with the following properties:
\begin{itemize}
    \item[(i)] It is a cocycle for the Leibniz algebra $\Omega^{n-2}(M)$.
    \item[(ii)] It vanishes on $d\Omega^{n-2}_c(M)$.
\end{itemize}
\end{Lemma}
\begin{proof}
The extension exists by \ref{lem:sheaf} in Appendix B and is unique because the compactly supported forms are dense in all forms. We recall how the extension is constructed: Given $\alpha \in  \Omega^{n-2}(M)$ and compact $K \subseteq M$, any $f_{K} \in C^{\infty}_{c}(M)$ with $f_K|_{K} = 1$ yields the same continuous linear functional 
$\widehat{\psi}(f_{K}\alpha) \in \Omega^{n-2}_{K}(M)'$, where $\Omega^{n-2}_{K}(M)$ denotes the subspace of forms supported in $K$. 
This defines an element  in the continuous linear dual of the injective limit $\varinjlim \Omega^{n-2}_{K}(M)=\Omega^{n-2}_{c}(M)$,
set to be the image of $\alpha$ by \eqref{doi}. We can now verify the properties of  $\widehat\psi$:

\begin{itemize}
\item [(i)] For $\alpha,\beta\in\Omega^{n-2}(M)$, $\gamma\in\Omega_K^{n-2}(M)$ and function $f_K$ as above, we get:
\begin{align*}
\widehat\psi([\alpha,\beta])\gamma&=\psi(f_K[\alpha,\beta],\gamma)=\psi([f_K\alpha,f_K\beta],\gamma)=\psi(f_K\alpha,[f_K\beta,\gamma])-\psi(f_K\beta,[f_K\alpha,\gamma])\\
&=\psi(f_K\alpha,[\beta,\gamma])-\psi(f_K\beta,[\alpha,\gamma])
=\widehat\psi(\alpha)[\beta,\gamma]-\widehat\psi(\beta)[\alpha,\gamma]=(\alpha\cdot\widehat\psi(\beta)-\beta\cdot\widehat\psi(\alpha))\gamma.
\end{align*}
The computation uses the fact that $\psi$ is diagonal at step two, since $f_K[\alpha,\beta]-[f_K\alpha,f_K\beta]$ vanishes on $K$,
thus its support is disjoint from $\supp(\gamma)\subseteq K$.
\item[(ii)] We show that $\widehat\psi(d\beta)=0$ for $\beta\in\Omega^{n-3}(M)$. Let $K$ be the support of $\gamma\in\Omega_c^{n-3}(M)$. 
By construction $\widehat\psi(d\beta)(\gamma)=\widehat\psi(f_Kd\beta)(\gamma)$.
Now with the identity $f_Kd\beta=d(f_K\beta)-(df_K)\wedge \beta$ we obtain that
$\widehat\psi(d\beta)(\gamma)=\widehat\psi(d(f_K\beta))(\gamma)- \widehat\psi(df_K\wedge \beta))(\gamma)=0$,
where the first term vanishes because $\widehat\psi$ vanishes on $d\Omega_c^{n-3}(M)$, while the second vanishes because $df_K$ is zero on the support of $\gamma$ and $\widehat\psi$ is support-decreasing.
\end{itemize}
\end{proof}

\subsection{Local triviality}

This section is devoted to the proof of the local version of the fact that 
$\oline{\Omega}{}_c^{n-2}(M):={\Omega_c^{n-2}(M)}/{d\Omega_{c}^{n-3}(M)}$ is centrally closed.
Thus we consider  the case where $M$ is a contractible open subset $U\subseteq \R^n$ and 
$\mu$ is the canonical volume form. In this case the following Lie algebras coincide,
\begin{equation}\label{three}
\oline{\Omega}{}_c^{n-2}(U)=\ceX(U,\mu)=\fX_c(U,\mu),
\end{equation}
and we prove that their second continuous cohomology group is zero.
In the next section we will use this to prove the global case, namely that $\oline{\Omega}{}^{n-2}_{c}(M)$ is centrally closed for arbitrary manifolds $M$.

We denote by $\fX_{\leq k}(U, \mu)$ the vector space of divergence-free vector fields with polynomial coefficients of degree at most $k$, 
and by $\fX_{ k}(U, \mu)$ the subspace with homogeneous ones of degree $k$. 
In particular 
\begin{equation}\label{patru}
[\fX_{ k}(U, \mu),\fX_{ l}(U, \mu)]\subseteq \fX_{ k+l-1}(U, \mu).
\end{equation}
Low degree cases are $\fX_{ 0}(U, \mu)=\R^n$ and $\fX_{ 1}(U, \mu)=\mathfrak{sl}(n,\R)$.
Notice that, under the above identification, the  $\mathfrak{sl}(n,\R)$-representation on $\fX_{k}(U, \mu) \subset \fX_{k}(U)=S^{k}(\R^{n})^*\otimes \R^n$ 
given by the Lie bracket in \eqref{patru} coincides with the natural action of $\mathfrak{sl}(n,\R)$ on this tensor product.

\begin{Proposition}\label{prop:workhorse}
    Let $n\geq 3$, and let $U\subset \mathbb R^n$ be a contractible open subset, and let $\psi\in CL^{2}(\Omega^{n-2}_c(U),\R)$ be a continuous skew-symmetric 
    Leibniz 2-cocycle. Then $\psi$ is a coboundary, i.e. $\psi=\dd\eta$ for some $\eta\in CL^{1}(\Omega^{n-2}_c(U),\R)=\Omega^{n-2}_c(U)'$.
\end{Proposition}

\begin{proof}
Let $\psi$ be the skew-symmetric continuous Leibniz 2-cocycle on $\Omega_c^{n-2}(U)$. The idea is to extend it to a differential operator 
$D \colon \fX(U) \rightarrow \Omega^{n-2}_{c}(U)'$
and use the vector spaces $\fX_{\leq k}(U, \mu)$ for an inductive proof that all the cocycles are coboundaries.

We extend $\psi$ to a Leibniz 1-cocycle $\widehat\psi$ on $\Omega^{n-2}(U)$ with values in $\Omega^{n-2}_{c}(U)'$, as in Lemma \ref{lemma:extendpsi}. Since $\psi$ is skew-symmetric, it vanishes on the ideal of squares $d\Omega^{n-3}_{c}(U) \subseteq \Omega^{n-2}_{c}(U)$. 
Since $\psi$ is continuous, $\widehat{\psi}$ vanishes on $d\Omega^{n-3}(U) \subseteq \Omega^{n-2}(U)$, and $\widehat\psi\circ d=0$.
Since $n\geq 3$, this allows to apply  the Poincar\'e lemma for differential operators, Lemma~\ref{Lemma:exactDO} in Appendix~\ref{Appendix:PoincareLemma},
to obtain that $\widehat\psi$ is of the form $Q \circ d$ for a differential operator 
$Q \colon \Omega^{n-1}(U) \rightarrow \Omega^{n-2}_{c}(U)'$.

The identification $\mu^{\flat} \colon \fX(U) \xrightarrow{\sim} \Omega^{n-1}(U)$ by 
the volume form is $C^{\infty}(U)$-linear, so $Q$ yields the differential operator $D:=Q \circ \mu^{\flat}$. 
By Peetre's Theorem (\ref{thm:peetre} in Appendix~\ref{Appendix:PoincareLemma}),  $D\colon \fX(U) \rightarrow \Omega^{n-2}_{c}(U)'$ admits a locally finite expansion
\begin{equation}\label{trei}
D\Big(\sum_{i=1}^nX^i\partial _i\Big) = \sum_{i=1}^n\sum_{\vec{\sigma} \in \mathbb{N}^{n}}
(\partial_{\vec\sigma}X^i)T^{\vec{\sigma}}_{i}
\end{equation}
in terms of $T^{\vec{\sigma}}_{i} \in \Omega^{n-2}_{c}(U)'$. 

If $X \in \fX(U, \mu)$ is divergence-free, then there exists a potential $\alpha_{X} \in \Omega^{n-2}(U)$ such that $i_{X}\mu = d\alpha_{X}$ 
(because $H^{n-1}_{\mathrm{dR}}(U) = 0$).
Since $D(X) = Q(i_{X}\mu) = Q(d\alpha_{X}) = \widehat{\psi}(\alpha)$, the restriction of $D$ to the divergence-free vector fields is a Lie algebra 1-cocycle
for $\fX(U, \mu)$ with values in $\Omega_c^{n-2}(U)'$, that is, 
\begin{equation}\label{eq:OneC}
D([X,Y]) = X \cdot D(Y) - Y \cdot D(X),\quad \forall X,Y\in  \fX(U, \mu)
\end{equation}
for the action $(X\cdot T)(\beta)=-T(L_X \beta)$ of $\fX(U,\mu)$ on $\Omega_c^{n-2}(U)'$.

Now we are ready to prove inductively that the 1-cocycle $D: \fX(U, \mu) \to\Omega_c^{n-2}(U)'$ is cohomologous to a 1-cocycle that vanishes on $\fX_{\leq k}(U,\mu)$. 

\paragraph{Step 0}
We find $\eta\in\Omega^{n-2}_{c}(U)'$ such that  the 1-cocycle $D-\dd \eta$ vanishes on $\fX_{0}(U, \mu)$, i.e.~on constant vector fields.

We use the language of currents, so let us denote by $\cD'^{q}(U) := \Omega^{n-q}_{c}(U)'$ the space of currents of degree $q$.
Then the exterior derivative  $d:\cD'^{q}(U) \to \cD'^{q+1}(U) $ obeys a Poincar\'e Lemma for currents \cite[Section I.2]{Demailly07}.
We define $T:= T^{\vec{0}}_{i}\otimes dx^{i}$, an $(\R^n)^*$-valued current of degree~2,
which we consider as a $\wedge^{n-2}\R^n$-valued current of degree~1 by means of the 
isomorphism 
\[
\cD'{^2}(U)\otimes (\R^n)^*
\simeq
\wedge^{n-2}\R^n \otimes \cD'{}^{0}(U) \otimes (\R^n)^*
\simeq 
\wedge^{n-2}\R^n \otimes \cD'{}^{1}(U).
\]
The cocycle identity \eqref{eq:OneC}
for constant vector fields  $X = \partial_{i}$ and $Y = \partial_{j}$ in $ \fX_{0}(U, \mu)$ yields 
$\partial_{i} \cdot T^{\vec{0}}_{j} - \partial_{j}\cdot T^{\vec{0}}_{i} = 0$.
This translates to $d T= 0$, because the action of $\partial_{i}$ on $T^{\vec{0}}_{j} \in\Omega^{n-2}_{c}(U)'\simeq \cD'{}^{0}(U) \otimes \wedge^{n-2}\R^n$ corresponds to the  ordinary derivative of distributions.
By the Poincar\'e Lemma for currents, there exists a $\wedge^{n-2}\R^n$-valued $0$-current
$\eta \in \cD'{}^{0}(U) \otimes \wedge^{n-2}\R^n \simeq \Omega^{n-2}_{c}(U)'$ with 
$T= d\eta$. This means that $T^{\vec{0}}_{i} = \partial_{i}\eta$ for all coordinate directions $i$,
so by the expansion \eqref{trei} the 1-cocycle $D - \dd \eta$ vanishes on $\fX_{0}(U, \mu)$.

\paragraph{Step 1}
Suppose that $D$ vanishes on $\fX_{0}(U,\mu)$. We find $\eta\in\Omega^{n-2}_{c}(U)'$ such that  the 1-cocycle $D-\dd \eta$ vanishes 
on $\fX_{\leq 1}(U, \mu)$.

For $X \in \fX_{0}(U,\mu)$ and 
$Y \in \fX_{1}(U, \mu)$ we have $[X,Y] \in \fX_{0}(U,\mu)$, so
the cocycle identity~\eqref{eq:OneC} yields $X \cdot D(Y) = 0$. It follows that 
$D(Y)\in\Omega^{n-2}_{c}(U)'$ is a constant current of degree~2, for all $Y \in \fX_{1}(U, \mu)$. 
The subspace of constant currents $\wedge^{n-2}\R^n \subset \Omega^{n-2}_{c}(U)'$
is a subrepresentation for the Lie subalgebra $\fX_{1}(U, \mu)$ of linear divergence-free vector fields.
If we identify $\fX_{1}(U, \mu)$ with $\mathfrak{sl}(n,\R)$, then this subrepresentation
$\wedge^{n-2}\R^n$ is the $(n-2)$-fold wedge product of the defining representation of $\mathfrak{sl}(n,\R)$, as one would expect. 

The restriction of $D$ to  linear divergence-free vector fields is a 1-cocycle on 
$\mathfrak{sl}(n,\R)$ with values in the finite dimensional representation $\wedge^{n-2}\R^n$, and hence a coboundary by Whitehead's Lemma.
Thus there exists $\eta \in \wedge^{n-2}\R^n \subseteq \Omega^{n-2}_{c}(U)'$ with  $D(X) = X \cdot \eta$ for all $X\in\fX_{1}(U, \mu)$. 
Moreover, since $\eta$ is a constant current, $X\cdot \eta = 0$ for $X \in \fX_{0}(U, \mu)$.
We obtain that the 1-cocycle $D-\dd \eta$ vanishes on $\fX_{\leq 1}(U, \mu)$.

\paragraph{Step k} Suppose that $D$ vanishes on $\fX_{\leq k-1}(U)$ for $k\ge 2$.
Then it also vanishes on  $\fX_{\leq k}(U, \mu)$.

We have $[\fX_{ 0}(U, \mu),\fX_{ k}(U, \mu)]\subseteq \fX_{ k-1}(U, \mu)$ by \eqref{patru},
so the cocycle identity for $D$ yields 
$X\cdot D(Y) = 0$ for all $X\in \fX_{0}(U,\mu)$ and $Y \in \fX_{k}(U,\mu)$. It follows that the restriction of $D$ to $\fX_{k}(U,\mu)$
takes values in the subspace of constant currents $\wedge^{n-2}\R^{n} \subseteq \Omega^{n-2}_{c}(U)'$.
The cocycle identity applied to $X \in \fX_{1}(U, \mu)$ and $Y \in \fX_{k}(U, \mu)$
then reads $D([X,Y]) = X \cdot D(Y)$, 
so the restriction of $D$ to 
$\fX_{k}(U, \mu) \subset \fX_{k}(U)=S^{k}(\R^{n})^*\otimes \R^n$ with values in $\wedge^{n-2}\R^{n} $ is an intertwiner of 
$\mathfrak{sl}(n,\R)$-representations.
Since $S^{k}(\R^{n})^*\otimes \R^n$ decomposes as a direct sum of two irreducible representation 
\cite[Prop~15.25]{FultonHarris04}, 
the $\mathfrak{sl}(n,\R)$-subrepresentation $\fX_{k}(U,\mu) \subseteq S^{k}(\R^{n})^*\otimes \R^n$ must be irreducible.
Since $\wedge^{n-2}\R^n$ is irreducible as well,
but not an irreducible subrepresentation of $S^{k}(\R^{n})^*\otimes \R^n$ \cite[Prop~15.25]{FultonHarris04},
the intertwiner  $D \colon \fX_{k}(U,\mu) \rightarrow \wedge^{n-2}\R^{n}$ is zero, and $D$ vanishes on $\fX_{\leq k}(U, \mu)$. Since a differential operator is completely determined by its values on polynomials, it follows that the 1-cocycle $D \colon \fX(U,\mu) \rightarrow \fX_{c}(U,\mu)'$ is a coboundary, $D=\dd\eta$ with $\eta \in  \Omega^{n-2}_{c}(U)'$. So $\widehat{\psi} \colon \Omega^{n-2}(U) \rightarrow \Omega^{n-2}_{c}(U)'$ is a coboundary as well, 
and so is $\psi \colon \Omega_{c}^{n-2}(U) \times \Omega_{c}^{n-2}(U) \rightarrow \R$.
\end{proof}

The above immediately implies:

\begin{Theorem}
\label{thm:local}
The Lie algebra of compactly supported divergence-free vector fields on $\R^n$ endowed with canonical volume form is centrally closed, i.e.~$H^2(\fX_c(\R^n,\mu),\R)=0$ .
\end{Theorem}

\begin{proof}
We have to show that $H^2(\fX_c(\R^n,\mu),\R)=0$. Let $\phi$ be a continuous Lie algebra 2-cocycle on $\fX_{c}(\R^n,\mu)$. We can pull back $\phi$ to a skew-symmetric cocycle $\psi$ on $\Omega^{n-2}_c(\R^n)$ and apply Proposition \ref{prop:workhorse} to obtain a potential $\eta$.

Using the skew-symmetry of $\phi$ we get for all $\alpha\in\Omega_c^{n-2}(\mathbb R^n)$:
\begin{align}\label{eq:vanonex}
  0=\phi(X_\alpha,X_\alpha)
  =\psi(\alpha,\alpha)
  =-\eta([\alpha,\alpha]),  
\end{align}

thus $\eta$ vanishes on the ideal of squares of the Leibniz algebra $\Omega_c^{n-2}(\mathbb R^n)$, which is  $d\Omega^{n-3}_{c}(\R^n)$ by Theorem \ref{Thm:squares}.
This means that $\eta$ arises from an $\bar\eta\in\fX_c(\R^n,\mu)'$ and $\widehat\phi \colon \fX_{c}(\R^n, \mu) \rightarrow  \fX_{c}(\R^n, \mu)'$ is its coboundary: $\widehat\phi=\dd\bar\eta$.
In particular the 2-cocycle $\phi$ is a coboundary as well, and $H^2(\fX_c(\R^n,\mu),\R)=0$.
The result holds for $n=2$ as well, see \cite{JV16}.
\end{proof}

\subsection{Global triviality}\label{subs:globtriv}

In this section we will use the local triviality established in Proposition \ref{prop:workhorse} to show that $\oline{\Omega}{}_c^{n-2}(M)$ and $\oline{\Omega}{}^{n-2}(M)$ have trivial second cohomology. We start with the compactly supported case:

\begin{Theorem} \label{prop:mainc}Let $(M,\mu)$ be a smooth manifold of dimension $\geq 3$, equipped with a volume form. Then 
$\oline{\Omega}{}_c^{n-2}(M)$ is centrally closed, i.e. $H^2(\oline{\Omega}{}_c^{n-2}(M),\R)=0$.
\end{Theorem}

\begin{proof} Let $\phi\in C^2(\oline{\Omega}{}_c^{n-2}(M),\R)$ be a cocycle. 
For $\fL=\Omega_c^{n-2}(M)$, we have $\fL_{\Lie}=\oline{\Omega}{}_c^{n-2}(M)$
by Theorem \ref{Thm:squares}. Hence, Proposition \ref{prop:inj} implies that to get the exactness of $\phi$ it suffices to show that $\psi=\pi^*\phi\in CL^2(\Omega^{n-2}_c(M))$ is exact. By Lemma \ref{lem:2coc1coc}, it suffices to verify that $\widehat \psi\colon \Omega_c^{n-2}(M) \rightarrow \Omega^{n-2}_{c}(M)'$ is exact.
 
Let $\{U_i\}$ an open cover of $M$ by contractible coordinate neighborhoods.
The operators $\widehat \psi|_{U_i}\colon \Omega_c^{n-2}(U_i) \rightarrow \Omega^{n-2}_{c}(U_i)'$ vanish on exact forms. In other words they are skew-symmetric, i.e. Proposition \ref{prop:workhorse} is applicable to them. This means that  $\widehat\psi|_{U_i}$ are exact with potentials $\eta_i\in\Omega^{n-2}_c(U_i)'$.

For any $\alpha,\beta\in \Omega^{n-2}_c(U_i)$, we have  $\eta_i([\alpha,\beta])=\widehat\psi_{i}(\alpha)\beta$. We first observe that 
$\eta_i$ and $\eta_j$ agree on $\Omega_{c}^{n-2}(U_i \cap U_j)$:
Indeed, since $\Omega_{c}^{n-2}(U_i \cap U_j)$ is a perfect Leibniz algebra,
every $\alpha \in \Omega^{n-2}_{c}(U_i \cap U_j)$ can be written as 
$\alpha = \sum_{r=1}^{N}[\beta_r,\gamma_r]$ with 
$\beta_r, \gamma_r \in \Omega^{n-2}_{c}(U_i \cap U_j)$, yielding
\[
\eta_i(\alpha)=\sum_{r=1}^{N}\eta_i([\beta_r,\gamma_r])=
\sum_{r=1}^{N}\psi|_{U_i}(\beta_r)\gamma_r=
\sum_{r=1}^{N}\psi|_{U_i\cap U_j}(\beta_r)\gamma_r=
\sum_{r=1}^{N}\psi|_{U_j}(\beta_r)\gamma_r=\sum_{r=1}^{N}
\eta_{j}([\beta_r,\gamma_r])=\eta_j(\alpha).
\]
By \cite[Proposition A1]{janssensNiestijl2024}, the presheaf that assigns to every open set $U$ the distributions $\Omega^{n-2}_{c}(U)'$  
is a sheaf, hence the $\eta_i$ can be glued to an element $\eta\in \Omega^{n-2}_{c}(M)'$. We would like to show that $\dd\eta=\hat\psi$. To see this we can observe that $\dd\eta-\hat\psi$ is support decreasing. This already implies that $\dd\eta-\hat\psi$ is identically zero because by construction it vanishes on the subalgebras $\Omega^{n-2}_{c}(U_i) \subseteq \Omega^{n-2}_{c}(M)$ for all $i$. 

Finally, we observe that by the same calculation as in Equation \eqref{eq:vanonex}, $\eta$ vanishes on exact forms, i.e. it passes to the quotient as a primitive 
$\bar\eta\in \oline{\Omega}{}^{n-2}_c(M)'=C^1(\oline{\Omega}{}^{n-2}_c(M),\R)$ of $\phi$.
\end{proof}

Now we show that we can infer the statement without support conditions from the compactly supported one:

\begin{Theorem}\label{prop:main}
    Let $(M,\mu)$ be a smooth manifold of dimension $\geq 3$, equipped with a volume form. Then 
$\oline{\Omega}{}^{n-2}(M)$ is centrally closed, i.e. $H^2(\oline{\Omega}{}^{n-2}(M),\R)=0$.
\end{Theorem}

\begin{proof}
As before, any continuous two-cocycle on $\oline{\Omega}{}^{n-2}(M)$ induces a one-cocycle $\widehat\psi:\Omega^{n-2}(M)\to \Omega^{n-2}(M)'$. We can restrict the latter to a cocycle $\widehat\psi_c:\Omega^{n-2}_c(M)\to \Omega^{n-2}_c(M)'$. By Theorem~\ref{prop:mainc} the latter is exact with a primitive $\eta\in \Omega^{n-2}_c(M)'$.

We claim that $\widehat \psi$ has compact support, i.e.\ that there exists a compact set $K$ such that $\hat\psi(\alpha)=0$ if ${\supp(\alpha)\cap K}=\emptyset$. If this were not the case, there would be a sequence of disjoint open sets $\{U_i\}_{i\in \mathbb N}$ and forms $\alpha_i,\beta_i\in\Omega^{n-2}_c(U_i)$ with $\psi(\alpha_i,\beta_i)=1$. Since $\psi \colon \oline{\Omega}{}^{n-2}(M) \times \oline{\Omega}{}^{n-2}(M) \rightarrow \R$ is diagonal and continuous, $\widehat\psi(\alpha)(\beta) = \psi(\alpha,\beta)$ would have infinite value on $\alpha=\sum \alpha_i$ and $\beta=\sum\beta_i$, which would form a contradiction.

The support of $\eta$ is then contained in $K$ as well. Indeed, any $x\notin K$ has an open neighbourhood $U$ such that 
$\psi(\oline{\Omega}{}^{n-2}_{c}(U), \oline{\Omega}{}^{n-2}_{c}(U)) = \{0\}$. Then 
\[
	\eta\big(\oline{\Omega}{}^{n-2}_{c}(U)\big) = \eta\big([\oline{\Omega}{}^{n-2}_{c}(U), \oline{\Omega}{}^{n-2}_{c}(U)]\big) =  
	\psi\big(\oline{\Omega}{}^{n-2}_{c}(U), \oline{\Omega}{}^{n-2}_{c}(U)\big) = \{0\},
\] 
so $x\notin \mathrm{supp}(\eta)$.
 Let $f$ be any compactly supported function with $f|_K=1$. We can define $\tilde\eta: \Omega^{n-2}(M)\to \R$ by $\tilde\eta(\alpha):=\eta(f\alpha)$. By construction, on compactly supported forms $\tilde \eta$ and $\eta$ coincide. This means that $\widehat \psi - \dd \tilde \eta$ vanishes on compactly supported forms. Since compactly supported forms  are dense in $\oline{\Omega}{}^{n-2}(M)$ and since $\widehat \psi - \dd \tilde \eta$ is continuous,
 this means that $\widehat \psi - \dd \tilde \eta$ is identically zero, i.e.\ $\widehat\psi$ is a coboundary.
\end{proof}

\section{Universal central extensions of Lie algebras}\label{sec:results}
In this section we apply the general theory developed in \cite{Neeb2002} together with the results from Section \ref{subs:globtriv}
 in order to construct universal central extensions for the Lie algebra $\mathfrak X_{\ex}(M,\mu)$ of exact divergence-free vector fields, 
 and for the Lie algebra $\ceX(M,\mu)$ of exact divergence-free vector fields that admit a compactly supported potential.
 
 Note that $\mathfrak X_{\ex}(M,\mu)$ fits into an exact sequence 
 \[
\fX_{\ex}(M,\mu) \hookrightarrow \fX(M,\mu) \twoheadrightarrow H^{n-1}_{\dR}(M)
 \]
 of locally convex Lie algebras, where $H^{n-1}_{\dR}(M)$ is abelian and the perfect Lie algebra $\fX_{\ex}(M,\mu)$ is the commutator ideal of $\fX(M,\mu)$ 
 (cf.~\cite{Lichnerowicz1974}).
In particular, $\mathfrak X(M,\mu)$ is perfect if and only if $H^{n-1}_{\dR}(M) = \{0\}$. In the algebraic setting, a Lie algebra has a universal central extension if and only if it is perfect \cite[\S1]{vdKallen1973}, so in general
one cannot expect $\mathfrak X(M,\mu)$ to have a universal central extension.
This explains our focus on the exact divergence-free vector fields.

In the same vein, the exact sequence  
 \[
	\ceX(M,\mu) \hookrightarrow \fX_{c}(M,\mu) \twoheadrightarrow H^{n-1}_{c, \dR}(M)
 \]
explains why in general one expects a universal central extension of $\ceX(M,\mu)$ and not of $\fX_{\ex}(M,\mu)$; 
by $ \cite{Lichnerowicz1974}$ the perfect Lie algebra $\ceX(M,\mu)$ is the commutator ideal of $\fX_{c}(M,\mu)$.

\subsection{Central extensions of locally convex Lie algebras}
Before we derive the main results, we briefly recall some properties of central extensions of topological Lie algebras from \cite{Neeb2002}.

Let $\mathfrak g$ be a topological Lie algebra. A central extension of $\mathfrak g$ is a continuous surjection $q':\mathfrak g'\to\mathfrak g$ of topological Lie algebras, such that $\mathfrak z':=ker(q')$ is central in $\mathfrak{g}'$.
In such cases we say that $\fg'$ is an extension of $\mathfrak g$ by $\mathfrak z'$.
A morphism of central extensions from $(q',\mathfrak g')$ to another central extension $(q'',\mathfrak g'')$ is a Lie algebra morphism $f:\mathfrak g'\to\mathfrak g''$ mapping $\mathfrak{z}'$ to $\mathfrak z''$ such that $q''f=q'$. 
Note that $f$ maps $\mathfrak{z}'$ into $\mathfrak{z}''$. 
If $\mathfrak{z}' = \mathfrak{z}''$, then a morphism is called strict if $f|_{\mathfrak{z}'} \colon \mathfrak{z}' \rightarrow \mathfrak{z}''$ is the identity.

A central extension $(q',\mathfrak g')$  is called linearly split, if there is a continuous linear section of $q'$. In this case, $\mathfrak g'\cong \mathfrak g\oplus \mathfrak z'$ as a topological vector space, and the bracket is uniquely characterized by a continuous two-cocycle of $\mathfrak g$ with values in 
the trivial $\fg$-module $\mathfrak z'$ (as in Section \ref{3.1}).
Cohomologous two-cocycles correspond exactly to isomorphic central extensions, hence we have a bijective correspondence:

\begin{quote}
linearly split central extensions of $\mathfrak g$ by $\mathfrak z'$ up to strict isomorphism ~~~~$\overset{1:1}{\Longleftrightarrow}$ ~~~~$H^2(\mathfrak g,\mathfrak z')$
\end{quote}

A central extension $\hat{\mathfrak z}\to \hat{\mathfrak g}\to \mathfrak g$ is called universal for a topological vector space $\mathfrak z'$, if for all linearly split central extensions $\mathfrak z'\to \mathfrak g'\to\mathfrak g$, there is a unique morphism of central extensions from $\hat{\mathfrak g}$ to $\mathfrak g'$. We underline here, that in principle a central extension could be universal for a certain class of central extensions (e.g. finite-dimensional $\mathfrak z'$) without being universal for another class of central extensions (e.g. locally convex ones). Let us try to describe universality in terms of cohomology:

\begin{Lemma}\label{lem:closedimpliesuniversal} Let $\hat q: \hat{\mathfrak g}\to \mathfrak g$ be a central extension of complete locally convex Lie algebras with:
	\begin{enumerate}
		\item $\hat{\mathfrak g}$ is perfect
		\item $H^2(\hat{\mathfrak g},\mathfrak z')=0$ where $\mathfrak z'$ is a locally convex vector space.
	\end{enumerate}
	Then $\hat{\mathfrak g}\to \mathfrak g$ is universal for central extensions by $\mathfrak z'$.
\end{Lemma}

\begin{proof} Let $\mathfrak z'\to \mathfrak g'\to \mathfrak g$ be a linearly split central extension. Let $\omega$ be the two-cocycle defining the bracket on $\mathfrak g'=\mathfrak g\oplus_{\omega} \mathfrak z'$. We can pull back the cocycle to obtain $\hat q^*\omega$ and correspondingly a central extension $\hat{\mathfrak g}\oplus_{\hat q^*\omega}\mathfrak z'$ of $\hat{\mathfrak g}$ by $\mathfrak z'$. Since $H^2_c(\hat{\mathfrak g},\mathfrak z')=0$, this extension is isomorphic to the trivial one and there exists a continuous Lie algebra morphism $f:\hat{\mathfrak g}\to \hat{\mathfrak g}\oplus_{\hat q^*\omega}\mathfrak z'$ that is the identity on the first component. We compose this morphism with the natural projection $(\hat q, \mathrm{id}):\hat{\mathfrak g}\oplus_{\hat q^*\omega}\mathfrak z'\to {\mathfrak g}\oplus_{\omega}\mathfrak z'$ to obtain a morphism of central extensions $F:\hat{\mathfrak{g}} \to {\mathfrak g}\oplus_{\omega}\mathfrak z'$. 
	
We claim this morphism is unique. So let us consider any morphism of central extensions $\tilde F:\hat{\mathfrak{g}} \to {\mathfrak g}\oplus_{\omega}\mathfrak z'$. Let $x\in\hat {\mathfrak g}$. Due to perfectness, $x=\sum[ y_i,z_i]$ for some finite collection $y_i,z_i$ of elements of $\hat{\mathfrak{g}}$. We have
$$
\tilde F(x)=\sum \tilde F([y_i,z_i])= \sum [\tilde F(y_i),\tilde F(z_i)]=\sum [F(y_i),F(z_i)]=F(x)
$$
Here we were allowed to replace $[\tilde F(y_i),\tilde F(z_i)]$ by $[F(y_i), F(z_i)]$, since the image of $F - \tilde{F}$ is central.
\end{proof}
\subsection{\texorpdfstring{Universal central extension of $\fX_{\ex}(M,\mu)$ and $\ceX(M,\mu)$}{Universal central extensions}}
Using these results from \cite{Neeb2002} combined with Theorems \ref{prop:mainc} and \ref{prop:main},
we readily obtain universal central extension of $\fX_{\ex}(M,\mu)$ and $\ceX(M,\mu)$.

\begin{Theorem}\label{Thm:UniversalNoncompactLA} Let $M$ be a smooth manifold of dimension at least $3$, and let $\mu$ be a volume form. Then the central extension 
	$$
	H^{n-2}_{\dR}(M)\to \oline{\Omega}{}^{n-2}(M)\to \mathfrak X_{\ex}(M,\mu)
	$$
	is 
	universal for linearly split extensions by complete locally convex spaces.
\end{Theorem}

\begin{proof} Since $\oline{\Omega}{}^{n-2}(M)$ is perfect, by Lemma \ref{lem:closedimpliesuniversal} we only need to show that $H^2(\oline{\Omega}{}^{n-2}(M), \mathfrak z')=0$ for any complete locally convex space $\mathfrak z'$. Since $\oline{\Omega}{}^{n-2}(M)$ is Fréchet, this follows directly from \cite[Proposition 2.10]{Neeb2002} and the fact that  $H^2( \oline{\Omega}{}^{n-2}(M), \mathbb R)=0$ as established in Theorem \ref{prop:main}.
\end{proof}

Applying this to the central extension $\R \oplus_{\omega}\fX_{\ex}(M,\mu)$ associated to a continuous $2$-cocycle $\omega$ on 
$\fX_{\ex}(M,\mu)$ with trivial coefficients, we conclude that every such 2-cocycle
is of the form 
\begin{equation}
\omega_{\lambda}(X_{\alpha}, X_{\beta}) = \lambda({L_{X_{\alpha}}\beta})
\end{equation}
for a continuous linear functional $\lambda \in \Omega^{n-2}(M)'$ with $\lambda(d\Omega^{n-3}(M)) = 0$, that is a closed compactly supported current of degree 2. A closed compactly supported 2-form $\sigma$ induces such a current $\lambda$, by $\lambda(\gamma)=\int_M \gamma\wedge \sigma$. The corresponding cocycle 
\begin{align}\label{eq:LichnerowiczCocycle}
    \omega_\lambda(X_\alpha,X_\beta)= \int_M(\iota_{X_\alpha}\iota_{X_\beta}\mu)\wedge \sigma= \int_M\sigma(X_\alpha,X_\beta)\mu
\end{align}
is called Lichnerowiz cocycle (cf \cite{Lichnerowicz1974}).  Alternatively a (smooth) closed cycle $C$ in singular homology induces a current $\lambda$ by $\lambda(\gamma)=\int_C\gamma$ (cf. \cite[Section 2.1]{JanssensVizmanSurvey2016}). The cocycle then reads: 
\begin{align}\label{eq:OtherCocycle}
    \omega_\lambda(X_\alpha,X_\beta)= \int_C\iota_{X_\alpha}\iota_{X_\beta}\mu.
\end{align}

If two functionals $\lambda$ and $\lambda'$ agree on $d\Omega^{n-3}(M)$, 
then their difference $\eta = \lambda - \lambda'$ defines a linear functional on $\Omega^{n-2}(M)/d\Omega^{n-3}(M) \simeq \fX_{\ex}(M,\mu)$.
So $\omega_{\lambda}$ is cohomologous to $\omega_{\lambda'}$, and we obtain the following result:

\begin{Corollary}
The map $\lambda \mapsto \omega_{\lambda}$ induces a linear isomorphism 
$H^{n-2}_{\dR}(M)' \stackrel{\sim}{\rightarrow} H^2(\fX_{\ex}(M,\mu),\R)$.
\end{Corollary}

Recall that $\ceX(M,\mu)$ is the Lie algebra of exact divergence-free vector fields that admit a compactly supported potential.
Since the Lie algebras $ \oline{\Omega}{}^{n-2}_c(M)$ and $\ceX(M,\mu)$ are not Fréchet when $M$ is non-compact, 
we obtain universality only for finite-dimensional extensions:

\begin{Theorem} Let $M$ be a smooth manifold of dimension at least $3$, and let $\mu$ be a volume form. Then the central extension 
	$$
	H^{n-2}_{c, \dR}(M)\to \oline{\Omega}{}_c^{n-2}(M)\to \ceX(M,\mu)
	$$
	is 
	universal for extensions by finite-dimensional spaces.
\end{Theorem}

\begin{proof}Since $H^2(\hat{\mathfrak g},\cdot)$ commutes with finite direct sums, Theorem \ref{prop:mainc} implies that  $H^2(\hat{\mathfrak g},\mathfrak z)=0$ for any finite-dimensional space $\mathfrak z$. The statement of the theorem then follows from Lemma \ref{lem:closedimpliesuniversal}.
\end{proof}
Similar to the reasoning above, the map $\lambda \mapsto \omega_{\lambda}$ (now for $\lambda \in \Omega^{n-2}_{c}(M)'$ with 
$\lambda(d\Omega_{c}^{n-3}(M))=0$) yields a linear isomorphism 
\begin{equation}
H^{n-2}_{c,\dR}(M)'\stackrel{\sim}{\rightarrow} H^2(\ceX(M,\mu),\R).
\end{equation}

\subsection{\texorpdfstring{The second cohomology of $\fX(M,\mu)$ and $\fX_c(M,\mu)$}{Second cohomology of divergence-free vector fields}}

To compute the second continuous cohomology group for the Lie algebra of divergence-free vector fields $\fX(M,\mu)$, we use the 
short exact sequence of Fr\'echet Lie algebras
\begin{equation}\label{short}0 \rightarrow \fX_{\ex}(M,\mu) \rightarrow \fX(M,\mu) \rightarrow H^{n-1}_{\dR}(M) \rightarrow 0.\end{equation}
Here $ \fX_{\ex}(M,\mu)$ is a perfect ideal of $\fX(M,\mu)$ and $H^{n-1}_{\dR}(M)$ is endowed with the trivial Lie bracket. 
\begin{Theorem} Let $M$ be a smooth manifold of dimension at least 3, and let $\mu$ be a volume form. Then
\[H^2(\fX(M,\mu),\R) = H_{\dR}^{n-2}(M,\R)' \oplus \Lambda^2 H_{\dR}^{n-1}(M,\R)'.
\]
\end{Theorem}
\begin{proof}
Let $\psi' \colon \fX(M,\mu) \times \fX(M,\mu) \rightarrow \R$ be a continuous 2-cocycle. Then its restriction $\psi$ to  $\fX_{\ex}(M,\mu)$ 
is cohomologous to a Lichnerowicz cocycle $\omega$ of the form \eqref{eq:LichnerowiczCocycle}, $\psi = \omega + \dd \lambda$ for a continuous 
linear functional $\lambda \colon \fX_{\ex}(M,\mu) \rightarrow \R$. 
If we extend $\omega$ to $\fX(M,\mu)$ by the same formula $\omega'(X,Y) := \int_M\sigma(X,Y)\mu$, then the result is still a cocycle \cite{VizmanPapucHonour}:
\[\textstyle
\omega'([X,Y], Z]) + \mathrm{cycl.} = \int_{M}\sigma([X,Y], Z])\mu + \mathrm{cycl.} = \int_{M}(L_{X}(\sigma(Y,Z)) \mu + \mathrm{cycl.} =0.
\]
By the Hahn--Banach Theorem for locally convex vector spaces, $\lambda$ extends to a continuous linear functional $\lambda' \colon \fX(M,\mu) \rightarrow \R$. 
Then $\gamma' := \psi' - \omega' - \dd\lambda'$ is a 2-cocycle on $\fX(M,\mu)$ that vanishes on $\fX_{\ex}(M,\mu) \times \fX_{\ex}(M,\mu)$.
Since $\fX_{\ex}(M,\mu) \subseteq \fX(M,\mu)$ is a perfect ideal, $\gamma'$ vanishes on $\fX_{\ex}(M,\mu) \times \fX(M,\mu)$
by the cocycle identity:
\[\gamma'(\fX_{\ex}(M,\mu), \fX(M,\mu)) = \gamma'([\fX_{\ex}(M,\mu), \fX_{\ex}(M,\mu)], \fX(M,\mu)) \subseteq 
\gamma'(\fX_{\ex}(M,\mu), [\fX_{\ex}(M,\mu), \fX(M,\mu)]) = \{0\}.
\]
So $\gamma'$ descends to a continuous 2-cocycle on the abelian Lie algebra $\fX(M,\mu)/\fX_{\ex}(M,\mu) = H^{n-1}_{\dR}(M,\mu)$. 
\end{proof}

A similar argument works in the compactly supported setting, using the short exact sequence of LF Lie algebras
\begin{equation}
0 \rightarrow \ceX(M,\mu) \rightarrow \fX_c(M,\mu) \rightarrow H^{n-1}_{c,\dR}(M) \rightarrow 0.
\end{equation}
We obtain
\begin{equation}
H^2(\fX_c(M,\mu) ,\R)
=H_{c,\dR}^{n-2}(M)'\oplus\Lambda^2H_{c,\dR}^{n-1}(M)'. 
\end{equation}

\section{Universal central extensions of Lie groups}\label{sec:LieGroups}

Let $M$ be a compact 3-manifold with an integral volume form $\mu$.
Then $\Diff(M,\mu)$ is a Fr\'echet--Lie group with Lie algebra $\fX(M,\mu)$ \cite{Hamilton1982}, 
and the flux homomorphism 
\[\mathrm{flux}_{\mu} \colon \fX(M,\mu) \rightarrow H^2_{\dR}(M): X \mapsto [i_{X}\mu]\]
integrates to a smooth Lie group homomorphism 
$\mathrm{Flux}_{\mu} \colon \Diff(M,\mu)_{0} \rightarrow J^2(M)$ 
from the connected identity component $\Diff(M,\mu)_{0}$ of $\Diff(M,\mu)$ to the Jacobian torus 
$J^2(M) = H^{2}_{\dR}(M)/(H^2(M,\Z)\otimes_{\Z}\R)$.
The connected identity component of the kernel of the flux  homomorphism is the group $\Diff_{\ex}(M,\mu)$ of exact volume preserving diffeomorphisms. It is a Fr\'echet--Lie group \cite{Gloeckner2006}, \cite[Thm.~III.11]{NeebWagemann2008} 
with Lie algebra $\fX_{\ex}(M,\mu)$ \cite{NeebVizman2003},
\cite[Prop.~3.8]{DiezJanssensNeebVizman2021}.

In joint work with Peter Kristel \cite{janssensHowActionThat2024}, 
we have constructed a Fr\'echet--Lie group extension $\widehat{\Diff}(M,\mu)_{\cG}$
of $\Diff_{\ex}(M,\mu)$ that covers the universal Lie algebra extension 
$\oline{\Omega}{}^1(M) \rightarrow \fX_{\ex}(M,\mu)$.
Using Neeb's Recogniton Theorem \cite{Neeb2002}, 
we show that, up to covering, this is the universal central extension of the group 
$\Diff_{\ex}(M,\mu)$ of exact volume preserving diffeomorphisms.

\subsection{Diffeomorphisms that stabilize a bundle gerbe}

The construction of the Lie group extension uses the 2-category of bundle gerbes.
Since the 3-form $\mu$ is integral, there exists a $\U(1)$-bundle gerbe with connection $\cG$ on $M$ whose curvature is~$\mu$.
Bundle gerbes with this property are in general not unique; they form a torsor over $H^2(M,\U(1)) = \mathrm{Hom}(H_2(M,\Z), \U(1))$.

By 
\cite[Rk.~3.4]{DiezJanssensNeebVizman2021},
a volume-preserving diffeomorphism $\phi\in \Diff(M,\mu)$ is exact if it lies in the connected identity component $\Diff(M,\mu)_{0}$ of $\Diff(M,\mu)$,
and if $\phi^*\cG$ is isomorphic to $\cG$ in the 2-category of bundle gerbes with connection \cite{Waldorf2007},
\begin{equation}
\Diff_{\ex}(M,\mu) = \{\phi \in \Diff(M,\mu)_{0}\,;\, \phi^*\cG \simeq \cG\}.
\end{equation}

If $\phi^*\cG$ is isomorphic to $\cG$, then the 1-morphisms $A \colon \phi^*\cG \rightarrow \cG$ are in general not unique;
their equivalence classes $\overline{A}$ modulo 2-morphisms constitute a torsor over $H^1(M,\U(1)) = \mathrm{Hom}(H_1(M,\Z),\U(1))$.
Equipped with the multiplication 
\[
(\phi, \oline{A})\cdot (\psi, \oline{B}) = (\phi\circ \psi, \oline{B\circ \psi^*A}), 
\]
the group 
\begin{equation}
\widehat{\Diff}_{\ex}(M,\mu)_{\cG} = \{(\phi, \overline{A})\,;\, \phi \in \Diff_{\ex}(M,\mu) \text{ and } A \in \mathrm{Hom}(\phi^*\cG, \cG)\}
\end{equation}
becomes a central extension of $\Diff_{\ex}(M,\mu)$ by $H^1(M,\U(1))$. In \cite{janssensHowActionThat2024} it is shown that 
$\widehat{\Diff}_{\ex}(M,\mu)_{\cG}$ is a Fr\'echet--Lie group in a natural way, and that the central extension of Fr\'echet--Lie groups
\begin{equation}
H^1(M,\U(1)) \stackrel{i}{\longrightarrow} \widehat{\Diff}_{\ex}(M,\mu)_{\cG}  \stackrel{q}{\longrightarrow} \Diff_{\ex}(M,\mu)
\end{equation}
gives rise to the universal central extension of Lie algebras
\begin{equation}
H_{\dR}^1(M,\R) \rightarrow \oline{\Omega}{}^1(M) \rightarrow \fX_{\ex}(M,\mu).
\end{equation}
Note that the centre $H^1(M,\U(1))$ is connected if and only if $H_1(M,\Z)$ is torsion-free: if
\[\textstyle{H_1(M,\Z) \simeq \Z^{b_1} \times \prod_{i=1}^{k}(\Z/ n_i\Z)}\] 
with $b_1$ the first Betti number of $M$, then 
 $H^1(M,\U(1)) \simeq \U(1)^{b_1} \times \prod_{i=1}^{k}C_{n_i}$ with $C_{n_i} \subseteq \U(1)$ the cyclic group of order $n_i$.

\subsection{\texorpdfstring{The universal central extension of $\Diff_{\ex}(M,\mu)$ for closed 3-manifolds}{The universal central group extension for closed 3-manifolds}}

Since $\widehat{\Diff}_{\ex}(M,\mu)_{\cG}$ need not be simply connected, it is not in general the universal central extension of $\Diff_{\ex}(M,\mu)$.
However, we will use Theorem~\ref{Thm:UniversalNoncompactLA} and the Recognition Theorem~\cite{Neeb2002} to show the following:
\begin{Theorem}\label{Thm:MainTheoremGroups}
If $M$ is a compact 3-manifold with integral volume form $\mu$, then
the universal cover of the connected identity component of $\widehat{\Diff}_{\ex}(M,\mu)_{\cG}$ is the universal central extension of $\Diff_{\ex}(M,\mu)$.
\end{Theorem}

We will use the following characterization of the universal central extension of a simply connected Fr\'echet--Lie group $H$, which was also employed in \cite[Theorem 7.1]{janssensHowActionThat2024}:
\begin{Theorem}[Theorem 4.13 in \cite{Neeb2002}]\label{neebcor}
	Consider a central extension $Z\to \Check{H}\to H$ of a Fréchet Lie group with finite-dimensional $Z$. Suppose that:
	\begin{enumerate}
		\item The Lie algebra $\mathfrak h$ is perfect,
		\item $H^2_c(\Check{\mathfrak h},\mathbb R)=0$,
		\item $\Check{H}$ is simply connected,
	\end{enumerate}
	then $Z\to \Check{H}\to H$  is universal for central extensions of $H$ by regular abelian Lie groups modeled on sequentially complete locally convex spaces.
\end{Theorem}
\begin{proof}
We used here a formulation equivalent to \cite[Theorem 7.1]{janssensHowActionThat2024} with the modification that condition 3) there is omitted since the perfectness of $\mathfrak h$ implies that the commutator group mentioned in this condition is the whole group $\tilde H$. 
\end{proof}

For brevity we write $G=\Diff_{\ex}(M,\mu)$, $\widehat{G}=\widehat{\Diff}_{\ex}(M,\mu)_{\cG}$ and $Z := H^1(M,\U(1))$ in the following.
For a Fr\'echet--Lie group $H$ with a smooth morphism $F \colon H \rightarrow G$, 
we denote the pullback by $\widehat{H} := F^*\widehat{G}$. Then
\begin{equation}\label{eq:bullbackCE}
\widehat{H} = \{(\hat{g},h) \in  \widehat{G} \times H\,;\,  q(\hat{g}) = F(h)\}
\end{equation}
is a central extension of $H$ by $Z$, denoted
\begin{equation}\label{eq:baseq}
Z\overset{j_H}{\longrightarrow}\widehat{H}\overset{q_H}{\longrightarrow} H.
\end{equation}
Finally, we will denote the universal cover of a Fr\'echet--Lie group $H$ by
\[
\pi_1(H)\overset{i_H}{\longrightarrow}\tilde H\overset{p_H}{\longrightarrow} H.
\]

\begin{Lemma}\label{lem:exx}
If $H$ is a 1-connected Fr\'echet--Lie group, then
$\tilde{\hat H}_0$ is a central extension of $H$. Moreover, the new central extension reads
    $$\tilde Z_0\overset{k_H}{\longrightarrow} \tilde {\hat H}_0\overset{q_{H}\circ p_{\hat H}}{\longrightarrow} H,$$
    where $\widetilde{Z}_0 \simeq H^1_{\dR}(M)$. 
\end{Lemma}

\begin{proof}
We start by an element $[\gamma]\in \tilde{\hat H}_0$ which is in the kernel of the projection to $H$. The path $\gamma$ starts from $e$ and since it has to project to $e$ in $H$, it ends in a point in $Z=ker(q_H)$. So $q_H\circ\gamma$ is a loop in $H$, hence contractible. Since $q_H$ is a fiber bundle, we can lift the homotopy between $q_H\circ\gamma$ and the trivial path to a homotopy in $\hat H_0$. This means $\gamma$ is homotopic to a path inside the connected identity component $Z_0$ of $Z$, 
and the kernel of the projection is the universal cover of $Z_{0} \simeq \U(1)^{b_1}$, which is isomorphic to $\R^{b_1} \simeq H^1_{\dR}(M)$.
 A path in $Z_0$ commutes with everything because the multiplication on the universal cover is defined pointwise. 
\end{proof}

\begin{Lemma} Let $M$ be compact 3-dimensional and $\mu$ an integral volume form. Then $\widehat{\tilde G} = \hat G\times_G\tilde G$, with 
    \begin{equation}\label{eq:pullbackeqn}
    \widehat G\times_G\widetilde G=\{(a,b)\in \widehat G\times \widetilde G\,;\, q_G(a)=p_G(b)\}
    \end{equation}
    the fibred product of $\widetilde{G}$ and $\widehat{G}$ over $G$. Further,
     \begin{equation}
     Z\times \pi_1(G)\to \widehat G\times_G\widetilde G \overset{p_G\circ q_{\tilde G}}{\longrightarrow} G
     \end{equation}
    is a central extension of $G$ by $Z\times \pi_1(G)$.
    \end{Lemma}
\begin{proof} 
Equation~\ref{eq:pullbackeqn} follows from \eqref{eq:bullbackCE}, and the extension is central because 
$Z\times \pi_1(G)$ is central in $\widehat G\times \widetilde G$. 
\end{proof}
We also have the following statement:
\begin{Corollary}
The universal covers of $\widehat{G}_0$ and $\widehat{\tilde{G}}_0$ are the same,
 
    $\widetilde{(\hat{\tilde{G}})}_0={\tilde{\hat{G}}}_0 = \widetilde{(\hat{G}_0\times_G\tilde G)}$.
\end{Corollary}

\begin{proof}
By the above Lemma, the projection $\widehat{\tilde G}\to \hat G$ has the fiber $\pi_1(G)$. In particular the fiber is discrete, i.e.\ it is a covering map. Since 
the connected component $(\widehat{\tilde G})_0$ is connected, this means that the simply connected universal cover of $(\widehat{\tilde G})_0$ is a universal cover of $\hat G_0$.
\end{proof}

We are now ready to prove Theorem~\ref{Thm:MainTheoremGroups}:

\begin{proof}[Proof of Theorem~\ref{Thm:MainTheoremGroups}]
Let $\widehat{G}_0$ be the connected identity component of $\widehat{G}$, and let $Z_r := Z \cap \widehat{G}_0$.
Then $Z_r \rightarrow \widehat{G}_{0} \rightarrow G$ is a connected central extension.
In the following commutative diagram, all rows and columns are exact:
\[
\begin{tikzcd}
		& 1 				& 1 											& 1 						& 			&   		\\
\pi_2(G) 	& \pi_1(Z_0) 		& \pi_1(\widehat{G}_{0}) 							& \pi_1(G) 				&\pi_0(Z_r) 	& 1		\\
1 		& \widetilde{Z_0}	&	\widetilde{\hat{G}}_0 = \widetilde{\hat{\tilde{G}}_0}	&\widetilde{G}				& 1			&		\\
1 		& Z_r			&	\widehat{G}_0								&		G				& 1			&		\\
\pi_1(G)	& \pi_0(Z_r)		&	1										&		1				& 			&		\\
		& 1				&											&						& 			&		
\arrow[from=2-1, to=2-2]
\arrow[from=2-2, to=2-3]
\arrow[from=2-3, to=2-4]
\arrow[from=2-4, to=2-5]
\arrow[from=2-5, to=2-6]
\arrow[from=3-1, to=3-2]
\arrow[from=3-2, to=3-3]
\arrow[from=3-3, to=3-4]
\arrow[from=3-4, to=3-5]
\arrow[from=4-1, to=4-2]
\arrow[from=4-2, to=4-3]
\arrow[from=4-3, to=4-4]
\arrow[from=4-4, to=4-5]
\arrow[from=5-1, to=5-2]
\arrow[from=5-2, to=5-3]
\arrow[from=1-2, to=2-2]
\arrow[from=2-2, to=3-2]
\arrow[from=3-2, to=4-2]
\arrow[from=4-2, to=5-2]
\arrow[from=5-2, to=6-2]
\arrow[from=1-3, to=2-3]
\arrow[from=2-3, to=3-3]
\arrow[from=3-3, to=4-3]
\arrow[from=4-3, to=5-3]
\arrow[from=1-4, to=2-4]
\arrow[from=2-4, to=3-4]
\arrow[from=3-4, to=4-4]
\arrow[from=4-4, to=5-4]
\end{tikzcd}
\]

\paragraph{Columns are exact}
The first column is exact because the kernel of $Z_r \mapsto \pi_i(Z_r)$ is the connected component $Z_0$ of $Z_r$, and 
because the universal cover of $Z_0$ yields an exact sequence
$1 \rightarrow \pi_1(Z_0) \rightarrow \widetilde{Z_0} \rightarrow Z_0 \rightarrow 1$.
The second and third columns are exact because they correspond to the universal covers of $\widehat{G}_0$ and $G$, respectively.
\paragraph{Rows are exact}
The first row is exact because it is part of the long exact sequence in homotopy corresponding to the Serre fibration $\widehat{G}_{0} \rightarrow G$ with fibre $Z_r$.
The second row is exact by Lemma \ref{lem:exx} applied to $H=\tilde G$, yielding the sequence 
\[\widetilde{Z_0} \to \tilde{\hat{\tilde G}}_0 \to \tilde G.\] 
Compatibility with the other rows follows from the identification ${\tilde{\hat{ G}}_0=\tilde{\hat{\tilde G}}}_0$. The third row is the connected component of 
the central extension $\widehat{G} \rightarrow G$. The fourth row is a shifted version of the first row.

\paragraph{The extension is central}
We need to check that the kernel of the diagonal map from $\widetilde{\hat{G}}_0 = \widetilde{\hat{\tilde{G}}_0}$ to $G$ is central.
Suppose that $x \in \widetilde{\hat{G}}_0$ maps to $1\in G$. Then a diagram chase to the left lower corner of the diagram yields an element 
$[\gamma] \in \pi_1(G)$. Now $\pi_1(G)$ occurs in the top right part of the diagram as well as a central subgroup of $\widetilde{G}$, and as such 
it gives rise to a central element $[\gamma] \in \widehat{G}$.
But the second row is a central extension by Lemma \ref{lem:exx} applied to $H = \widetilde{G}$, 
so there exists a \emph{central} element $x_{\gamma} \in \widetilde{\hat{G}}_0$ that maps to $[\gamma]$ under the map 
$\widetilde{\hat{G}}_0 \rightarrow \widetilde{G}$.

In order to show that $x$ is central in $\widetilde{\hat{G}}_0$, it therefore suffices to show that $x (x_{\gamma})^{-1}$ is central.
Let $z_{\gamma} \in Z_r \subseteq \widehat{G}_0$ be the image of $x_{\gamma}$ under $\widetilde{\hat{G}}_0 \rightarrow \widehat{G}_0$.
Then $\pi_0(z_{\gamma}) \in \pi_0(Z_r)$ is precisely the image of $[\gamma]$ under the map $\pi_1(G) \rightarrow \pi_0(Z_r)$
that comes from the Serre fibration $\widehat{G}_0 \rightarrow G$. 
Indeed, the diagram chase from $\pi_1(G)$ on the right upper part of the diagram to $\pi_0(Z_r)$ in the left lower corner proceeds by
taking a closed loop in $G$ that starts and ends at the identity, lifting it to a path in $\widehat{G}_{0}$ that starts at the identity and ends in $Z_r$, 
and taking the connected component of the fibre $Z_r$ of $\widehat{G}_{0} \rightarrow G$ determined by the end point. 

It follows that if we replace $x$ by $x (x_{\gamma})^{-1}$, the diagram chase from $\widetilde{\hat{G}}_0$ to the left lower corner 
$\pi_1(G)$ yields the identity. So $x (x_{\gamma})^{-1}$ is in the image of $\widetilde{Z_0}$, which is central in $\widetilde{\hat{G}}_0$.

This shows that the extension $\widetilde{\hat{G}}_0 \rightarrow G$ is central. Since the corresponding Lie algebra extension 
is still $\oline{\Omega}{}^1(M) \rightarrow \mathfrak X_{\ex}(M,\mu)$, the Recognition Theorem~\ref{neebcor} is applicable,
and we conclude that the central extension is universal. 
\end{proof}

\begin{Remark}
Since the construction of $\widehat{\Diff}_{\ex}(M,\mu)$ crucially depends on a fusion product on loop space (which does not appear to have an analogue for higher bundle gerbes), our construction of the universal central extension is currently restricted to manifolds $M$ of dimension 3. 
However, for $\mathrm{dim}(M)\geq 3$, there does exist a procedure, cf.\ \cite{HallerVizman2004} and \cite[Section 5.3.1 and 5.3.2]{DiezJanssensNeebVizman2021}, 
to construct central $\U(1)$-extensions of 
$\Diff_{\ex}(M,\mu)$
that integrate cocycles of the form \eqref{eq:LichnerowiczCocycle} for 
integral $[\sigma] \in H^{2}(M, \R)_{\Z}$,
as well as cocycles of the form 
\eqref{eq:OtherCocycle} 
for integral classes $[C] \in H^{n-2}(M,\Z)$
that can be represented by a smooth submanifold.
\end{Remark}

\appendix

\section{A multivector field description of the de Rham complex}\label{app:multivfcartan}

It will be convenient to identify $\alpha \in \Omega^{n-k}(M)$ with the multivector field $A \in \Gamma(\wedge^kTM)$ using  the volume form $\mu$. If $\alpha = \iota_{A}\mu$, we write $\alpha=A^{\flat}$ and $A=\alpha^{\sharp}$.  Our sign conventions for  contraction of multivector fields into a form are fixed by
$\iota_{A\wedge B}\mu = \iota_{B}\iota_{A}\mu$. With this identification, the de Rham differential $d$ on $\Omega^{\bullet}(M)$ and the Leibniz bracket on $\Omega^{n-2}(M)$  give rise to a differential $\delta(A) := d(A^{\flat})^{\sharp}$ on $\Gamma(\wedge^{\bullet}TM)$ and  and a Leibniz bracket $[A,B] := [A^{\flat}, B^{\flat}]^{\sharp}$ on $\Gamma(\wedge^{2}TM)$.

\begin{Proposition}
The differential $\delta \colon \Gamma(\wedge^kTM) \rightarrow \Gamma(\wedge^{k-1}TM)$ is given by
\begin{eqnarray*}
\delta (X_1 \wedge \cdots \wedge X_k)
&=& \sum_{1\leq i < j \leq k} (-1)^{k + i + j }[X_i, X_j] \wedge X_1 \wedge \cdots  \hat{X}_i \cdots \hat{X}_j \cdots \wedge X_k\\
& &+ \sum_{i=1}^{k}(-1)^{k + i} \Div(X_i) X_1 \wedge \cdots \hat{X}_i \cdots \wedge X_k.
\end{eqnarray*}
\end{Proposition}
\begin{proof}
The case $k=1$ is the definition of divergence. The case $k$ follows from $k-1$ using \[d \iota_{X_k} (\iota_{X_{1}\wedge \cdots \wedge X_{k-1}}\mu) = 
L_{X_k} \iota_{X_{1} \wedge \cdots \wedge X_{k-1}}\mu - \iota_{X_k} d(\iota_{X_{1} \wedge \cdots \wedge X_{k-1}}\mu)\]
and 
$
L_{X_k} \iota_{X_{1} \wedge \cdots \wedge X_{k-1}}\mu = 
\Div(X_k)\iota_{X_{1} \wedge \cdots \wedge X_{k-1}} \mu - \sum_{j=1}^{k-1}\iota_{X_{1} \wedge \ldots \wedge [X_j, X_k]\wedge  \cdots \wedge X_{k-1}}\mu .
$
\end{proof}

\begin{Corollary}\label{Corollary:deltaBivector}
If $\alpha = \iota_{X_1\wedge X_2}\mu$, then $X_{\alpha} = \delta(X_1\wedge X_2)$ is given by
\begin{align}\label{eq:xalp}
     X_{\alpha} = \Div(X_2)X_1 - \Div(X_1)X_2 - [X_1,X_2].
\end{align}
 \end{Corollary}

Similarly, the Leibniz bracket on $\Gamma(\wedge^2TM)$ takes the following form.  
\begin{Proposition}\label{Prop:bracketbivectors}
The Leibniz bracket on $\Gamma(\wedge^2TM)$ is given by 
\begin{equation*}
[X_1\wedge X_2, Y_1 \wedge Y_2] = [\delta(X_1 \wedge X_2), Y_1]\wedge Y_2 + Y_1 \wedge [\delta(X_1 \wedge X_2), Y_2],
\end{equation*}
or, equivalently, by
\begin{eqnarray*}
[X_1\wedge X_2, Y_1 \wedge Y_2] 
&=&- [[X_1,X_2], Y_1]\wedge Y_2 - Y_1 \wedge [[X_1,X_2], Y_2]\\
& & - \Div(X_1) [X_2,Y_1]\wedge Y_2 - \Div(X_1) Y_1 \wedge [X_2,Y_2]\\
& & + (L_{Y_1}\Div(X_1)) X_2 \wedge Y_2 + (L_{Y_2}\Div(X_1))Y_1 \wedge X_2\\
& & + \Div(X_2) [X_1,Y_1]\wedge Y_2 + \Div(X_2) Y_1 \wedge [X_1,Y_2]\\
& & - (L_{Y_1}\Div(X_2)) X_1 \wedge Y_2 - (L_{Y_2}\Div(X_2)) Y_1 \wedge X_1.
\end{eqnarray*}
\end{Proposition}
\begin{proof}
If $\alpha = (X_1 \wedge X_2)^{\flat}$ and $\beta = (Y_1 \wedge Y_2)^{\flat}$, then 
$[\alpha, \beta] = L_{X_{\alpha}}\beta = L_{X_{\alpha}} \iota_{Y_2}\iota_{Y_1}\mu$.
Using $L_{X_{\alpha}} \iota_{Y_{j}} = \iota_{[X_{\alpha}, Y_{j}]} + \iota_{Y_{j}}L_{X_{\alpha}}$ for $j\in \{1,2\}$ and $L_{X_{\alpha}} \mu = 0$, we find 
\[
	[\alpha, \beta] = \iota_{[X_{\alpha}, Y_1]\wedge Y_2 + Y_1 \wedge [X_{\alpha}, Y_2]}\mu.
\]
Substituting \eqref{eq:xalp} then yields the required result.
\end{proof}
In particular, if $\alpha = \iota_{X_1 \wedge X_2}\mu$ and $\beta = \iota_{Y_1 \wedge Y_2}\mu$, then $[\alpha,\beta]$ is obtained 
by inserting $[X_1\wedge X_2, Y_1 \wedge Y_2]$ into~$\mu$.

\section{A compactly supported Poincaré Lemma with parameters}\label{sec:appendixParameterPoincare}

In the proof of Lemma \ref{lem:cubesquares}, we needed a version of the compactly supported Poincaré Lemma with additional parameters. While the existence of such a Poincaré Lemma is intuitively clear, we could not find any reference with a compactly supported version, so we provide a proof here for the sake of completeness. We supply here an elementary geometric proof for a cube, an alternative approach would be to construct a parametrized version of the support-preserving Poincaré Lemma of \cite{Bogovskii1980SolutionOS} (cf. also \cite{nutzi2024supportpreservinghomotopyrham}).

Let $M,N$ be manifolds. We consider $X=M\times N$ with the foliation $F=TM\times 0_N$. The foliated longitudinal forms $\Omega^{\bullet,0}(X)=\Gamma(X,\Lambda^\bullet F^*)$ are exactly the complex of differential forms along $M$, and their differential is the de Rham differential $d_M$ in $M$ direction, with $N$ being treated as a parameter. We denote its cohomology $H^{\bullet,0}(X)=H^{\bullet,0}(M\times N)$. 
Alternatively we could see $F$ as a Lie algebroid (with the inclusion as the anchor) and the above cohomology is just the Chevalley-Eilenberg cohomology of this Lie algebroid. When $M$  admits a finite good open cover, the K\"unneth theorem for Lie algebroids (\cite[Theorem 6.6]{JotzMarchesini2024}) implies:
\begin{align*}
    H^{\bullet,0}(M\times N)=H^\bullet_{dR}(M)\otimes C^\infty(N).
\end{align*}
In the context of foliations, this statement goes back at least to \cite{ElKacimiAlaoui1983} (cf. also \cite{Bertelson2011}). This allows us to prove the following:
\begin{Lemma}\label{lem:PoincParam}
  Let  us consider a precompact cube $U\times V\subset \mathbb R^k\times \mathbb R^{n}$ for $k\geq 2$ and $n\geq 1$.   Given a  form $\alpha\in\Omega^{k,0}(U\times V)$ such that
\begin{itemize}
    \item $supp(\alpha)\subset U \times V$ is compact
    \item $\int_U\alpha=0$ as an element of $C^\infty(V)$,
\end{itemize} 

there exists a compactly supported form $\delta\in \Omega^{k-1,0}(U\times V)$ with $d_M\delta=\alpha$.
\end{Lemma}

\begin{proof}
We start by considering a slightly smaller cube $U''\times V'$ with $supp(\alpha)\subset U''\times V'$ such that $U''\subset U'\subset U$ and $V'\subset V$ are relatively compact sets. At the same time we see $\mathbb R^k$ as $S^k- N$, where $N$ is the north pole of the $k$-sphere. 
\begin{itemize}
\item The form $\alpha$ extends by zero to a form $\tilde \alpha\in \Omega^{k,0}(S^k\times V)$. The class of $\tilde \alpha$ is trivial since $\int_{S^k}\tilde\alpha=\int_U\alpha=0$. (Here we implicitly use  $H^{\bullet,0}(S^k\times V)=C^\infty(V)$ as follows from the above K\"unneth theorem.) Let $\beta\in \Omega^{k-1,0}(S^k\times V)$ be a primitive of $\tilde \alpha$.
\item Let  $\tilde\beta=\beta|_{S^k \backslash \overline{U''}\times V}$. We have $d_M\beta=\tilde\alpha|_{S^k \backslash \overline{U''}\times V}=0$. Since $H^{k-1,0}(S^k \backslash \overline{U''}\times V)=H^{k-1}(S^k \backslash \overline{U''})\otimes C^\infty(V)=0$, the form $\tilde\beta$ has a $d_M$-potential $\gamma\in \Omega^{k-2,0}(S^k \backslash\overline{ U''}\times V)$.

\item  Let $\rho_1\in C^{\infty}_c(S^k\backslash\overline{ U''})$ have compact support and be constantly 1 on $S^k\backslash U'$. Similarly let $\rho_2\in C^{\infty}_c(V)$ have compact support and be constantly 1 on a neighborhood of $V'$. We set $\delta=\rho_2\cdot(\beta-d_M(\rho_1\gamma))$.
\end{itemize}
By construction $d_M\delta=\rho_2d(\beta-d_M(\rho_1\gamma))=\rho_2\tilde \alpha=\tilde \alpha$, so we only have to understand why $\delta$ is compactly supported in $U\times V$. The term $(\beta-d_M(\rho_1\gamma))$ is supported in $U'\times V$ (since we have $(d_M\rho_1\gamma)|_{S^k\backslash \overline{U'}}=\beta|_{S^k\backslash \overline{U'}}$). Consequently $\delta$ is supported in $U'\times V'$ which is precompact in $U\times V$.
\end{proof}

\section{A Poincaré Lemma for differential operators}\label{Appendix:PoincareLemma}

\label{appendix:poincare}
In this section we will establish a Poincaré type Lemma for differential operators, which we need for the proof of Theorem \ref{thm:local}. First, we briefly recall Peetre's theorem for support-decreasing linear operators. 
For partial derivatives, we will use the notation 
$\partial_{\vec{\sigma}}f := (\frac{\partial}{\partial x_1})^{\sigma_1} \cdots (\frac{\partial}{\partial x_n})^{\sigma_n}f$
with $\vec{\sigma} = (\sigma_1, \ldots, \sigma_n) \in \N^{n}$ as usual.

\begin{Theorem}[Peetre \cite{peetreRectificationArticleCaracterisation1960}] \label{thm:peetre} Let $E,F$ be vector bundles over $M$ and let $P:\Gamma_c(E)\to \Gamma_c(F)'$ be a support-decreasing linear map. Then there exists a discrete set $\Lambda\subset M$ such that $P|_{M\backslash \Lambda}$ is continuous. 
Moreover, the restriction of $P$ to $M\setminus \Lambda$ is a differential operator of locally finite order: for any $p\in M\setminus \Lambda$, there 
exists a chart $(U,x)$ and frame $\{e_i\}$ of $E_U$ and finitely many nonzero distributions $T^{\vec{\sigma}}_{i}\in \Gamma_c(F|_U)'$
such that 
\[
P(s) =\sum_{i=1}^{\mathrm{rank}(E)} \sum_{\vec{\sigma} \in \mathbb{N}^{n}}
\left(\partial_{\vec{\sigma}} s^{i}\right) 
T^{\vec{\sigma}}_{i}
\]
for all $s=\sum s^ie_i \in \Gamma_{c}(E|_{U})$. 
\end{Theorem}

The original result \cite{peetreRectificationArticleCaracterisation1960} was stated for open subsets of $\R^n$ and for trivial line bundles, 
but the above version easily reduces to this because the statement is local. In detail:
\begin{proof}
Let $\mathcal{V}$ be a locally finite cover of $M$ such that $E$ and $F$ trivialise over every $V\in \mathcal{V}$.
Let $e_i$ and $f_j$ be the corresponding $C^{\infty}(V)$-bases of $\Gamma(E|_{V})$ 
 and $\Gamma(F|_{V})$, and write $s = \sum_i s^i e_i$ and $t = \sum_{j} t^j e_{j}$ for sections of $E|_{V}$ and $F|_{V}$, respectively.
 Then $P_{ij}(f)(g) := P(fe_i)(gf_j)$ is a support-decreasing linear map $P_{ij} \colon C_{c}^{\infty}(V) \rightarrow C^{\infty}_{c}(V)'$. By the original result
 \cite{peetreRectificationArticleCaracterisation1960}, there exists a discrete set $\Lambda^{V}_{ij}\subseteq V$ such that 
 $P_{ij}$ is continuous on $V\setminus \Lambda^{V}_{ij}$, and every $p\in V\setminus \Lambda^{V}_{ij}$ admits 
 a neighbourhood $U_{ij}$ such that 
 \[P_{ij}(f)(g) = \sum_{\vec{\sigma}\in \N^{n}} \left(\partial_{\vec{\sigma}}f\right) T^{\vec{\sigma}}_{ij}(g)\] 
 for all $f, g \in C^{\infty}_{c}(U_{ij}\setminus \Lambda^{V}_{ij})$. Then $P$ is continuous on $M\setminus \Lambda$ for 
 the discrete set $\Lambda = \bigcup \Lambda^{V}_{ij}$. For $p\in M\setminus \Lambda$ we can find a coordinate neighbourhood $U = \bigcap U_{ij}$ such that
 with $P(s)(t) = \sum_{ij}P_{ij}(s^i)(t^j)$ for all $s\in \Gamma_{c}(E|_{U})$ and $t\in \Gamma_{c}(F|_{U})$, and the result follows with 
 $T^{\vec{\sigma}}_{i}(t) = \sum_{j}T^{\vec{\sigma}}_{ij}(t^{j})$ for $t \in \Gamma_{c}(F|_{U})$.
\end{proof}

The above characterization allows us to formulate a coordinate-free description of (distribution-valued) differential operators:

\begin{Definition} 
We call a support-decreasing linear map $P:\Gamma_c(E)\to \Gamma_c(F)'$ a differential operator (of locally finite order), if $\Lambda=\emptyset$, i.e. if  $P$ is continuous.
\end{Definition}

Any support-decreasing continuous operator $P:\Gamma(E)\to \Gamma_c(F)'$ induces a differential operator by restriction to $\Gamma_c(E)\subset \Gamma(E)$. However for differential operators the opposite is also true:

\begin{Lemma}\label{lem:sheaf} Let $P:\Gamma_c(E)\to \Gamma_c(F)'$ be a differential operator. Then $P$ induces a continuous morphism of sheaves $\mathcal E_E\to \mathcal D_F'$, where $\mathcal E_E(U)=\Gamma(E|_U)$ and $\mathcal D_F'(U)=\Gamma_c(F|_U)'$. In particular it induces a continuous map $\Gamma(E)\to \Gamma_c(F)'$.
\end{Lemma}

\begin{proof}
This is stated in \cite{peetreRectificationArticleCaracterisation1960}, however we will provide here a short explanation. First of all for any $U\subset M$, $P$ can be restricted to a continuous operator
$$P_U:\Gamma_c(E|_U)\to \Gamma_c(F|_U)'$$
by the natural extension $\Gamma_c(E|_U)\to \Gamma_c(E)$ and restriction $\Gamma_c(F)'\to \Gamma_c(F|_U)'$. The resulting operator is still support-decreasing, hence we only need to show extendibility from $\Gamma_c(E|_U)$ to $\Gamma(E|_U)$. 

Given $s \in  \Gamma(E|_U)$ and compact $K \subseteq U$,
any $f_{K} \in C^{\infty}_{c}(U)$ with $f_K|_{K} = 1$ yields the same continuous linear functional 
$P_U(f_{K}s) \in \Gamma_K(F|_U)'$, where $\Gamma_K$ denotes sections with support in $K$.
This defines an element  in the continuous linear 
dual of the locally convex injective limit $\Gamma_{c}(F|_U) = \varinjlim \Gamma_{K}(F|_U)$, denoted again by $P_U(s)$. 
We get a support-decreasing operator
$P_U \colon \Gamma(E|_U) \rightarrow  \Gamma_c(F|_U)'$ that extends the originally given one. It is continuous because for every compact $K\subseteq M$, the composition of $P_U$ with the projection $\Gamma_{c}(F|_U)' \rightarrow \Gamma_{K}(F|_U)'$ is continuous.
\end{proof}
We can now turn to the Poincaré Lemma for differential operators:
\begin{Lemma}\label{Lemma:exactDO}
    Let $U\subset \mathbb R^n$ be connected and open, and let $E\to U$ be a vector bundle. For $k \geq 1$, let $D:\Omega^k_c(U)\to \mathcal E=\Gamma_c(E)'$ be a differential operator (of locally finite order) with $Dd=0$. 
 Then there exists a differential operator $Q:\Omega^{k+1}_c(U)\to\mathcal{E}$ such that $D=Qd$,
\[\begin{tikzcd}[column sep=large]
	\Omega^{k+1}_{c}(U)\arrow[dr, dashed, "Q"] & \\
	\Omega^{k+1}_{c}(U)\arrow[u, "d"]\arrow[r, "D"] & \Gamma_{c}(E)\\
	\Omega^{k-1}_{c}(U) \arrow[u, "d"]  \arrow[ur, "0"].&
\end{tikzcd}\]       
\end{Lemma}

\begin{proof}
We start by noting that $D$ uniquely extends from $\Omega^k_c(U)$ to $\Omega^k(U)$ by Lemma \ref{lem:sheaf}. We will work with this extension and note that it vanishes on all exact forms $d\Omega^{k-1}(U)$, since $\Omega^{k-1}_c(U)$ is dense in $\Omega^{k-1}(U)$.

Let $\Omega^k_{\ell}(U)$ denote the space of differential $k$-forms with polynomial coefficients of degree $\ell$ and  $\Omega^k_{\le\ell}(U)$ the space of differential $k$-forms with polynomial coefficients of degree $\le\ell$. We construct inductively differential operators $D_\ell:\Omega^{k}(U)\to \mathcal E$ of order $\ell$ and $Q_\ell:\Omega^{k+1}(U)\to \mathcal E$ of order $\ell-1$  such that 
    \begin{enumerate}
    \item $D_\ell=Q_\ell d$ 
    \item $D-D_\ell$ vanishes on $\Omega^k_{\le\ell}(U)$.
\end{enumerate}
We start with $D_0=0$ and correspondingly $Q_0=0$, because $D$ vanishes on differential $k$-forms with constant coefficients: $\Omega^k_{\le 0}(U)=\Omega^k_0(V)\subseteq d\Omega^{k-1}_1(U)$. 

Let $E = \sum_{\mu=1}^{n}x^\mu\partial_{\mu}$ be the Euler vector field.
Given a differential operator $D_\ell$ of order $\ell$ which satisfies 1.\ and 2., we define the following differential operators of order $\ell+1$:
    \begin{align*}
        D_{\ell+1}&:=D_{\ell}+\frac{1}{k+\ell+1}[(D-D_\ell)\iota_E]^{\le\ell}d.\\
        Q_{\ell+1}&:=Q_\ell+\frac{1}{k+\ell+1}[(D-D_\ell)\iota_E]^{\le\ell}.
    \end{align*}
Here $[A]^{\leq \ell}$ denotes the part of the differential operator $A$ which is of degree at most $\ell$.\footnote{This part can be extracted from the original differential operator coefficient by coefficient, by evaluating the operator on all monomials of degree $\leq l$.} This notion is not coordinate-invariant, but that does not cause any problem for the proof, since we work in a fixed coordinate system. By construction and induction hypothesis, Property 1., i.e. $D_{\ell+1}=Q_{\ell+1}d$, is satisfied. 
    
Let us verify Property 2. If $\alpha\in\Omega^{k}_{\le\ell+1}(U)$, then $d\alpha\in\Omega^{k+1}_{\le\ell}(U)$ and $\frac{1}{k+\ell+1}[(D-D_\ell)\iota_E]^{>\ell}d\alpha=0$. This means that $[(D-D_\ell)\iota_E]^{\le\ell}d\alpha$ is equal to $(D-D_\ell)\iota_E d\alpha$. Thus: 
$$(D-D_{\ell+1})\alpha=(D-D_\ell)\alpha-\frac{1}{k+\ell+1}(D-D_\ell)\iota_Ed\alpha=(D-D_\ell)\alpha-\frac{1}{k+\ell+1}(D-D_\ell)L_E\alpha,$$
where in the last equality we use the fact that $(D-D_\ell)d=Dd-D_{\ell}d=0$. Now Property 2 holds:
\begin{itemize}
	\item for $\alpha\in\Omega^k_{\ell+1}(U)$ because 
	$L_E\alpha=(k+\ell+1)\alpha$,
	\item and for $\alpha\in\Omega^k_{\le\ell}(U)$ because $L_E\alpha\in\Omega^k_{\le\ell}(U)$ and $D-D_\ell$ vanishes on the subspace $\Omega^k_{\le\ell}(U)$. 
\end{itemize} 

On any precompact open set $V\Subset U$ the order of $D$ is bounded by some $\ell$. There we have $D|_V=D_\ell|_V$, since a differential operator of order $\ell$ is completely determined by what it does on polynomials of degree $\leq l$. Consequently $Q_{\ell+i}|_V=Q_\ell|_V$ for all $i\geq 0$. This already implies that $Q=\lim_{\ell\to \infty} Q_\ell|_{\Omega^{k+1}_c(U)}$ is a support-decreasing operator $\Omega^{k+1}_c(U)\to \mathcal E$. Moreover $Q$ is continuous, since it is locally continuous, i.e. it is a differential operator.
\end{proof}

\bibliographystyle{plain}
\bibliography{bibliography}
\end{document}